\documentclass[a4paper]{amsart}
\usepackage{hyperref}
\usepackage[left=0.75in, right=0.75in, top= 1in, bottom= 1.75in ]{geometry}
\usepackage{amsmath,amssymb,amsthm,amsfonts}
\theoremstyle{definition}
\newtheorem{defn}{Definition}[section]

\newtheorem{thm}[defn]{Theorem}
\newtheorem{fact}[defn]{Fact}
\newtheorem{corr}[defn]{Corollary}
\newtheorem{lem}[defn]{Lemma}
\newtheorem{question}[defn]{Question}
\newtheorem{remark}[defn]{Remark}
\newtheorem{observation}[defn]{Observation}

\title[Combinatorial Properties and  $\mathsf{DC_{\kappa}}$ in symmetric extensions]{Combinatorial Properties and  Dependent Choice in symmetric extensions based on L\'{e}vy Collapse}
\author{Amitayu Banerjee}
\address{Alfr\'ed R\'enyi Institute of Mathematics, Reáltanoda utca 13-15, Budapest-1053, Hungary}
\email{banerjee.amitayu@gmail.com}
\date{}
\makeatletter
\@namedef{subjclassname@2020}{\textup{2020} Mathematics Subject Classification}
\makeatother
\begin{document}
\begin{abstract}
We work with symmetric extensions based on L\'{e}vy Collapse and extend a few results of Apter, Cody, and Koepke. We prove a conjecture of Dimitriou from her Ph.D. thesis.
We also observe that if $V$ is a model of $\mathsf{ZFC}$, then $\mathsf{DC_{<\kappa}}$ can be preserved in the symmetric extension of $V$ in terms of symmetric system $\langle \mathbb{P},\mathcal{G},\mathcal{F}\rangle$, if $\mathbb{P}$ is $\kappa$-distributive and $\mathcal{F}$ is $\kappa$-complete. Further we observe that if $\delta<\kappa$ and $V$ is a model of $\mathsf{ZF+DC_{\delta}}$, then $\mathsf{DC_{\delta}}$ can be preserved in the symmetric extension of $V$ in terms of symmetric system $\langle \mathbb{P},\mathcal{G},\mathcal{F}\rangle$, if $\mathbb{P}$ is ($\delta+1$)-strategically closed and $\mathcal{F}$ is $\kappa$-complete.
\end{abstract}
\maketitle
\section{introduction}
Grigorieff \cite{Gri1975} proved that symmetric extensions in terms of {\em symmetric system} $\langle \mathbb{P},\mathcal{G},\mathcal{F}\rangle$\footnote{ $\langle\mathbb{P},\mathcal{G},\mathcal{F}\rangle$ is a {\em symmetric system} if $\mathbb{P}$ is a forcing notion, $\mathcal{G}$ is  a group of automorphisms of $\mathbb{P}$, and $\mathcal{F}$ is a normal filter of subgroups over $\mathcal{G}$.} are intermediate models of the form HOD$(V[a])^{V[G]}$ as $a$ varies over $V[G]$. Apter, Henle, Cody, and Koepke constructed several models of $\mathsf{ZF}$ in terms of hereditarily definable sets based on L\'{e}vy Collapse (cf. \cite{AC2013, Apt1983a, AK2006, Apt2005, AH1991}). The purpose of this note is to translate the arguments of a few of those  choiceless model constructions to symmetric extensions in terms of symmetric system $\langle \mathbb{P},\mathcal{G},\mathcal{F}\rangle$ and extend a few published results. In particular, we prove the following. 

\begin{enumerate}
    \item We prove the failure of $\mathsf{AC_{\kappa}}$ (Every family of $\kappa$ non-empty sets admits a choice function) in the symmetric extension of \cite[\textbf{Theorem 4.1}]{Kar2019}. Moreover, we study an argument to preserve the supercompactness of $\kappa$ in the symmetric model following the methods of \cite{Ina2013}.
    
    \item We reduce the large cardinal assumption of \cite[\textbf{Theorem 2}]{AC2013} and \cite[\textbf{Theorem 3}]{AC2013}.
    \item We observe an {\em infinitary Chang conjecture} in the choiceless model constructed in  \cite[\textbf{Theorem 11}]{AK2006}. Moreover, we prove that $\aleph_{\omega_{1}}$ is an almost Ramsey cardinal in the model.

    \item Fix an arbitrary $n_{0}\in \omega$. We observe that if $\langle S_{k}: 1\leq k <\omega\rangle$ is a sequence of stationary sets such that $S_{k}\subseteq\aleph_{n_{0}+2(k+1)}$ for every $1\leq k <\omega$, then $\langle S_{k}: 1\leq k <\omega\rangle$ is {\em mutually stationary} in the choiceless model constructed in \cite[\textbf{Theorem 1}]{Apt1983a}. We also observe an alternating sequence of measurable and non-measurable cardinals in the model. Moreover, we observe that $\aleph_{\omega}$ is an almost Ramsey cardinal in the model.
\end{enumerate}

Secondly, we prove a conjecture of Dimitriou related to the failure of Dependent Choice--or $\mathsf{DC}$--in a symmetric extension based on finite support products of collapsing functions, from \cite{Dim2011}. We also study new lemmas related to preserving $\mathsf{DC}$ in symmetric extensions inspired by \cite[\textbf{Lemma 1}]{Kar2014}. In particular, we observe the following.
\begin{enumerate}
    \item Let $V$ be a model of $\mathsf{ZFC}$. If $\mathbb{P}$ is $\kappa$-distributive and $\mathcal{F}$ is $\kappa$-complete, then $\mathsf{DC_{<\kappa}}$ is preserved in the symmetric extension of $V$ with respect to the symmetric system $\langle \mathbb{P},\mathcal{G},\mathcal{F}\rangle$.
    \item Let $\delta<\kappa$ and $V$ be a model of $\mathsf{ZF+DC_{\delta}}$ where the Axiom of Choice ($\mathsf{AC}$) might fail. If $\mathbb{P}$ is ($\delta+1$)-strategically closed and $\mathcal{F}$ is $\kappa$-complete, then $\mathsf{DC_{\delta}}$ is preserved in the symmetric extension of $V$ with respect to the symmetric system $\langle \mathbb{P},\mathcal{G},\mathcal{F}\rangle$.
\end{enumerate}

\subsection{Preserving Dependent Choice} Karagila \cite[\textbf{Lemma 1}]{Kar2014} proved that if $\mathbb{P}$ is $\kappa$-closed and $\mathcal{F}$ is $\kappa$-complete then $\mathsf{DC_{<\kappa}}$ is preserved in the symmetric extension in terms of symmetric system $\langle \mathbb{P},\mathcal{G},\mathcal{F}\rangle$.
We observe that `$\mathbb{P}$ is $\kappa$-closed' can be replaced by `$\mathbb{P}$ is $\kappa$-distributive' in \cite[\textbf{Lemma 1}]{Kar2014}. This slightly generalize \cite[\textbf{Lemma 1}]{Kar2014}, since there are $\kappa$-strategically closed forcing notions which are not $\kappa$-closed\footnote{As for an example, the forcing notion $\mathbb{P}(\kappa)$ which adds a {\em non-reflecting stationary set of cofinality $\omega$ ordinals in $\kappa$}, is $\kappa$-strategically closed but not even $\omega_{2}$-closed. (cf. \cite[\textbf{section 6}]{Cum2010}).} and $\kappa$-distributivity is weaker than  $<\kappa$-strategic closure.\footnote{As for an example, the forcing notion for {\em killing a stationary subset of $\omega_{1}$}, is $\omega_{1}$-distributive but not even $<\omega_{1}$-strategically closed (cf. \cite[\textbf{section 6}]{Cum2010}).}

\begin{observation}{(\textbf{Lemma 3.2})}
{\em Let $V$ be a model of $\mathsf{ZFC}$. If $\mathbb{P}$ is $\kappa$-distributive and $\mathcal{F}$ is $\kappa$-complete, then $\mathsf{DC_{<\kappa}}$ is preserved in the symmetric extension of $V$ with respect to the symmetric system $\langle \mathbb{P},\mathcal{G},\mathcal{F}\rangle$.}
\end{observation}

We also observe that even if we start with a model $V$, which is a model of $\mathsf{ZF +DC_{\delta}}$ where $\mathsf{AC}$ might fail, we can still preserve $\mathsf{DC_{\delta}}$ in a symmetric extension of $V$ in certain cases. In particular, we observe the following.

\begin{observation}{(\textbf{Lemma 3.4})}
{\em Let $\delta<\kappa$ and $V$ be a model of $\mathsf{ZF + DC_{\delta}}$. If $\mathbb{P}$ is ($\delta+1$)-strategically closed and $\mathcal{F}$ is $\kappa$-complete, then  $\mathsf{DC_{\delta}}$ is preserved in the symmetric extension of $V$ with respect to the symmetric system $\langle \mathbb{P},\mathcal{G},\mathcal{F}\rangle$. }
\end{observation}

\subsection{On a question of Apter} Woodin asked in the context of $\mathsf{ZFC}$, that if $\kappa$ is strongly compact and $\mathsf{GCH}$ holds below $\kappa$, then must $\mathsf{GCH}$ hold everywhere? The problem is still open in the context of $\mathsf{ZFC}$. One variant of this question is if $\mathsf{GCH}$ can fail at every limit cardinal less than or equal to a strongly compact cardinal $\kappa$ where as $\mathsf{GCH}$ holds above $\kappa^{+}$. Apter answered this in the context of $\mathsf{ZF}$. Apter \cite[\textbf{Theorem 3}]{Apt2012} constructed a model where $\kappa$ is a regular limit cardinal and a supercompact cardinal, and $\mathsf{GCH}$ holds for a limit $\delta$ if and only if $\delta>\kappa$. In that model the Countable Choice--or $\mathsf{AC_{\omega}}$-- fails. At the end of \cite{Apt2012}, Apter asked the following question.

\begin{question}
{\em Is it possible to construct analogs of Theorem 3 in which some weak version of $\mathsf{AC}$ holds ?}
\end{question}

The author and Karagila constructed a symmetric extension to answer \textbf{Question 1.3} in \cite[\textbf{Theorem 4.1}]{Kar2019}. 

\begin{thm}{(cf. \cite[\textbf{Theorem 4.1}]{Kar2019})}
{\em Let V be a model of $\mathsf{ZFC + GCH}$ with a supercompact cardinal $\kappa$. Then there is a
symmetric extension in which $\mathsf{DC_{<\kappa}}$ holds, $\kappa$ is a regular limit cardinal and supercompact, and $\mathsf{GCH}$ holds for a limit cardinal $\delta$ if and only if $\delta >\kappa$.}
\end{thm}

For the sake of convenience, we call the symmetric extension constructed
in \cite[\textbf{Theorem 4.1}]{Kar2019}
as $\mathcal{N}_{1}$. We study an argument to preserve the supercompactness of $\kappa$ in $\mathcal{N}_{1}$ applying the methods of \cite{Ina2013} and prove the following in \textbf{section 4}. 

\begin{thm}
{\em In $\mathcal{N}_{1}$, $\mathsf{AC_{\kappa}}$ fails.}
\end{thm}

\subsection{Proving Dimitriou's conjecture} Dimitriou constructed a symmetric extension based on finite support products of collapsing functions. At
the end of \cite[\textbf{section 1.4}]{Dim2011}, Dimitriou conjectured that $\mathsf{DC}$ would fail in the symmetric extension (cf. \cite[\textbf{Question 1}, \textbf{Chapter 4}]{Dim2011}).
We prove the conjecture. For the sake of convenience, we call this model as Dimitriou's model and prove the following in \textbf{section 5}. 

\begin{thm}
{\em In Dimitriou's model, $\mathsf{AC_{\omega}}$ fails.}
\end{thm}

\subsection{Reducing the assumption of supercompactness to strong compactness} Apter and Cody \cite[\textbf{Theorem 2}]{AC2013} obtained a model of $\mathsf{ZF +\neg AC_{\omega}}$ where $\aleph_{1}$ and $\aleph_{2}$ are both singular of cofinality $\omega$, and there is a sequence of distinct subsets of $\aleph_{1}$ of length equal to any predefined ordinal, assuming a supercompact cardinal $\kappa$. In \textbf{section 6}, we observe that applying a recent result of Usuba (cf. \cite[\textbf{Theorem 3.1}]{ADU2021}) followed by working with a model of $\mathsf{ZF +\neg AC_{\omega}}$ constructed using {\em strongly compact Prikry forcing}%\footnote{An exhibition of symmetric extension on strongly compact Prikry forcing can be found in \cite{AH1991}.}
, it is possible to reduce the assumption of a supercompact cardinal $\kappa$ to a strongly compact cardinal $\kappa$.

\begin{observation}
{\em Suppose  that $\kappa$ is a strongly compact cardinal, $\mathsf{GCH}$ holds, and $\theta$ is an ordinal. Then there is a model of $\mathsf{ZF +\neg AC_{\omega}}$ in which $cf(\aleph_{1})=cf(\aleph_{2})=\omega$,
and there is a sequence of distinct subsets of $\aleph_{1}$ of length $\theta$.}
\end{observation}

Similarly, we reduce the large cardinal assumption of \cite[\textbf{Theorem 3}]{AC2013} from a supercompact cardinal to a strongly compact cardinal. Apter and Cody \cite[\textbf{Theorem 3}]{AC2013} obtained a model of $\mathsf{ZF +\neg AC_{\omega}}$ where $\aleph_{\omega}$ and $\aleph_{\omega+1}$ are both singular with $\omega\leq cf(\aleph_{\omega+1})<\aleph_{\omega}$, and there is a sequence of distinct subsets of $\aleph_{\omega}$ of length equal to any predefined ordinal, assuming a supercompact cardinal $\kappa$.
We prove the following in \textbf{section 6}. 

\begin{observation}
{\em Suppose that $\kappa$ is a strongly compact cardinal, $\mathsf{GCH}$ holds, and $\theta$ is an ordinal. Then there is a model of $\mathsf{ZF +\neg AC_{\omega}}$ in which $\aleph_{\omega}$ and $\aleph_{\omega+1}$ are both singular with $\omega\leq cf(\aleph_{\omega+1})<\aleph_{\omega}$, and there is a sequence of distinct subsets of $\aleph_{\omega}$ of length $\theta$.}
\end{observation}

\subsection{Infinitary Chang conjecture from a measurable cardinal} Assuming a measurable cardinal, Apter and Koepke constructed a model $\mathcal{N}$ of $\mathsf{ZF}$ based on L\'{e}vy collapse in \cite[\textbf{Theorem 11}]{AK2006}. In $\mathcal{N}$, $\omega_{1}$ is singular, 
and $\aleph_{\omega_{1}}$ is a Rowbottom cardinal carrying a Rowbottom filter. They mentioned that in $\mathcal{N}$, $\mathsf{AC_{\omega}}$ fails because of the singularity of $\omega_{1}$. In \textbf{section 7}, we observe an {\em infinitary Chang conjecture} in a symmetric extension, which is very similar to $\mathcal{N}$, except we consider a finite support product construction. We use the observation that it is possible to force a {\em coherent} sequence of Ramsey cardinals after performing Prikry forcing on a normal measure over a measurable cardinal $\kappa$ (cf. \cite[\textbf{Theorem 3}]{AK2006}). We also use the observation that an infinitary Chang conjecture can be established in a symmetric model, assuming a coherent sequence of Ramsey cardinals. As in the model of \cite[\textbf{Theorem 11}]{AK2006}, $\omega_{1}$ is singular and therefore $\mathsf{AC_{\omega}}$ fails.

\begin{thm}
{\em Let $V'$ be a model of $\mathsf{ZFC}$ where there is a measurable cardinal. Then there is a generic extension $V$ of $V'$, and a symmetric extension $V(G)$ of $V$ such that $\omega_{1}$ is singular in $V(G)$.
%and thus \textbf{\em AC$_{\omega}$} fails. 
Moreover, an infinitary Chang conjecture holds in $V(G)$.
}
\end{thm}

Similarly, we also observe an infinitary Chang conjecture in the model $\mathcal{N}$. For the sake of convenience, we call the model $\mathcal{N}$ as Apter and Koepke's model and prove the following in \textbf{section 7}.  

\begin{thm}
{\em An infinitary Chang conjecture holds in Apter and Koepke's model. Moreover, $\aleph_{\omega_{1}}$ is an almost Ramsey cardinal in the model.}
\end{thm}

\subsection{Mutual stationarity property from a sequence of measurable cardinals}
Foreman and Magidor \cite{FM2001} introduced the idea of {\em mutual stationarity}. They asked if there is a model of set theory in which every sequence of stationary subsets of the $\aleph_{n}$'s of a fixed cofinality is mutually
stationary (cf. \cite[page 290]{FM2001}). Assuming an $\omega$-sequence of supercompact cardinals, Apter \cite[\textbf{Theorem 1}]{Apt2005} constructed a model of $\mathsf{ZF}+\mathsf{DC}$ in which if $\langle S_{n} : 1 \leq n<\omega\rangle$ is a sequence of stationary sets such that $S_{n} \subseteq \aleph_{n}$ for every $1 \leq n <\omega$, then $\langle S_{n} : 1 \leq n<\omega\rangle$ is mutually stationary. Apter \cite[\textbf{Theorem 1}]{Apt1983a} further obtained a similar model based on L\'{e}vy collapse as constructed in \textbf{\cite{Apt2005}}, where $\aleph_{\omega}$ carries a Rowbottom filter and $\mathsf{DC_{\aleph_{n_{0}}}}$ holds for any arbitrary $n_{0}\in \omega$, from an $\omega$-sequence of measurable cardinals. For the sake of convenience, we fix an arbitrary $n_{0}\in \omega$ in the ground model $V$, call the model from \cite[\textbf{Theorem 1}]{Apt1983a} as $\mathcal{N}_{n_{0}}$, and prove the following in \textbf{section 8}. 

\begin{observation}
{\em The following hold in the model $\mathcal{N}_{n_{0}}$.

\begin{enumerate}
    \item For each $1\leq k<\omega$, $\aleph_{n_{0}+2(k+1)}$ is a measurable cardinal and $\aleph_{n_{0}+2k+1}$ is not a measurable cardinal. In particular, for each $1\leq k<\omega$, there are no unifrom ultrafilters on $\aleph_{n_{0}+2k+1}$.
    \item If $\langle S_{k}: 1\leq k <\omega\rangle$ is a sequence of stationary sets such that $S_{k}\subseteq\aleph_{n_{0}+2(k+1)}$ for every $1\leq k <\omega$, then $\langle S_{k}: 1\leq k <\omega\rangle$ is mutually stationary.
    \item $\aleph_{\omega}$ is an almost Ramsey cardinal.
\end{enumerate}
}
\end{observation}

\textbf{Structure of the paper.} 
\begin{itemize}
    \item In \textbf{section 2}, we cover the basics.
    \item In \textbf{section 3}, we prove \textbf{Observations 1.1 $\&$  1.2}.
    \item In \textbf{section 4}, we prove \textbf{Theorem 1.5}.
    
    \item In \textbf{section 5}, we prove \textbf{Theorem 1.6}. 
    
    \item In \textbf{section 6}, we prove \textbf{Observations 1.7  $\&$ 1.8}. 
    
    \item In \textbf{section 7}, we prove \textbf{Theorems 1.9 $\&$ 1.10}. 
    \item In \textbf{section 8}, we prove \textbf{Observation 1.11}.
\end{itemize}
 
\section{Basics}
\subsection{Large Cardinals}
In this section, we recall the definition of inaccessible cardinals in the context of $\mathsf{ZFC}$ and other large cardinals in the context of $\mathsf{ZF}$. In $\mathsf{ZFC}$, we say $\kappa$ is a strongly inaccessible cardinal if it is a regular strong limit cardinal where the definition of ``strong limit" is that for all $\alpha<\kappa$, we have $2^{\alpha}<\kappa$. In the context of $\mathsf{ZF}$, the above definition doesn't make sense, as $2^{\alpha}$ may not be well-ordered. We refer the reader to \cite{BDL2007} for details concerning inaccessible cardinals in the context of $\mathsf{ZF}$.
We recall some large cardinal definitions in the context of $\mathsf{ZF}$ from \cite{Kan2003}.

\begin{defn} Let $\kappa$ be an uncountable cardinal.
\begin{enumerate}

\item The cardinal $\kappa$ is {\em weakly compact} if for all $f:[\kappa]^{2}\rightarrow 2$, there is a homogeneous set $X\subseteq \kappa$ for $f$ of order type $\kappa$. 

\item The cardinal $\kappa$ is {\em Ramsey} if for all $f:[\kappa]^{<\omega}\rightarrow 2$, there is a homogeneous set $X\subseteq \kappa$ for $f$ of order type $\kappa$. 

\item The cardinal $\kappa$ is {\em almost Ramsey} if for all $\alpha<\kappa$ and $f:[\kappa]^{<\omega}\rightarrow 2$, there is a homogeneous set $X\subseteq \kappa$ for $f$ having order type $\alpha$.

\item  The cardinal $\kappa$ is {\em $\mu$-Rowbottom} if for all $\alpha<\kappa$ and $f:[\kappa]^{<\omega}\rightarrow \alpha$, there is a homogeneous set $X\subseteq\kappa$ for $f$ of order type $\kappa$ such that $\vert f^{''}[X]^{<\omega}\vert<\mu$. We say that $\kappa$ is Rowbottom if it is $\omega_{1}$-Rowbottom. A filter $\mathcal{F}$ on $\kappa$ is a {\em Rowbottom filter} on $\kappa$ if for any $f:[\kappa]^{<\omega}\rightarrow\lambda$, where $\lambda<\kappa$, there is a set $X\in\mathcal{F}$ such that $\vert f^{''}[X]^{<\omega}\vert\leq\omega$.

\item The cardinal $\kappa$ is {\em measurable} if there is a $\kappa$-complete free ultrafilter on $\kappa$. 
A filter $\mathcal{F}$ on a cardinal $\kappa$ is {\em normal} if it is closed under diagonal intersections:

\begin{center}
    If $X_{\alpha}\in \mathcal{F}$ for all $\alpha<\kappa$, then $\Delta_{\alpha<\kappa}X_{\alpha}\in \mathcal{F}$. 
\end{center}

In $\mathsf{ZF}$ we have the following lemma. 

\begin{lem}{(cf. \cite[\textbf{Lemma 0.8}]{Dim2011})}
{\em An ultrafilter $\mathcal{U}$ over $\kappa$ is normal if and only if for every regressive $f:\kappa\rightarrow\kappa$ there is an $X\in \mathcal{U}$ such that $f$ is constant on $X$.} 
\end{lem}

Thus, we say an ultrafilter $\mathcal{U}$ over $\kappa$ is normal if for every regressive $f:\kappa\rightarrow\kappa$ there is an $X\in \mathcal{U}$ such that $f$ is constant on $X$.

\item For a set $A$, we say $\mathcal{U}$ is a {\em fine measure} on $\mathcal{P}_{\kappa}(A)$ if $\mathcal{U}$ is a $\kappa$-complete ultrafilter and for any $i\in A$, $\{x\in\mathcal{P}_{\kappa}(A): i\in x\}\in\mathcal{U}$. We say that $\mathcal{U}$ is a {\em normal measure} on $\mathcal{P}_{\kappa}(A)$ if $\mathcal{U}$ is a fine measure and if $f:\mathcal{P}_{\kappa}(A)\rightarrow A$ is such that $f(X)\in X$ for a set in $\mathcal{U}$, then $f$ is constant on a set in $\mathcal{U}$. The cardinal $\kappa$ is {\em $\lambda$-strongly compact} if there is a fine measure on $ \mathcal{P}_{\kappa}(\lambda)$; it is {\em strongly compact} if it is $\lambda$-strongly compact for all $\kappa\leq\lambda$.

\item The cardinal $\kappa$ is {\em $\lambda$-supercompact} if there is a normal measure on $ \mathcal{P}_{\kappa}(\lambda)$; it is {\em supercompact} if it is $\lambda$-supercompact for all $\kappa\leq\lambda$.
\end{enumerate}
\end{defn}

\begin{remark}
We note that the definition of supercompact (similarly strongly compact) is given in the terms of ultrafilters, which is weaker than the definition of supercompact in terms of elementary embedding due to Woodin \cite[\textbf{ Definition 220}]{Wood2010} (e.g. $\aleph_{1}$ can be supercompact or strongly compact if we consider the definition of supercompact or strongly compact in terms of ultrafilters, but $\aleph_{1}$ can not be the critical point of an elementary embedding (cf. \cite{Ina2013})). 
\end{remark}

\begin{remark}
Ikegami and Trang \cite[\textbf{section 2}]{IT2019} defined that an ultrafilter $\mathcal{U}$ on $\mathcal{P}_{\kappa}X$ is normal if for any set $A \in \mathcal{U}$ and $f : A \rightarrow \mathcal{P}_{\kappa}X$
with $\emptyset \not= f(\sigma) \subseteq \sigma$ for all $\sigma\in A$, there is an $x_{0} \in X$ such that for $\mathcal{U}$-measure one many $\sigma$ in $A$, $x_{0}\in f(\sigma)$. They note that their definition of normality is equivalent to the closure under diagonal intersections in $\mathsf{ZF}$, while it may not be equivalent to the definition of normality in our sense without $\mathsf{AC}$.
\end{remark}

From now on, all our inaccessible cardinals are strongly inaccessible. We recall that a limit of Ramsey cardinals is an almost Ramsey cardinal in $\mathsf{ZF}$  (cf. \cite[\textbf{Proposition 1}]{AK2008}). 

\subsection{Forcing extension and L\'{e}vy--Solovay Theorem} 
Let $\mathbb{P}$ be a forcing notion, by which we mean a partially ordered set with a
maximum element 1.  For $p, q \in \mathbb{P}$, we say that $p$ is stronger than $q$ or $p$ extends $q$ if $p \leq q$. Let $G$ be a $\mathbb{P}$-generic filter over $V$ and
$V^{\mathbb{P}}$ be a class of all $\mathbb{P}$-names defined recursively as follows: if $\tau$ is a set, $\tau \in V^{\mathbb{P}}$ if and only if $\tau \subseteq V^{\mathbb{P}} \times \mathbb{P}$. The interpretation of a $\mathbb{P}$-name $\tau$ by $G$ is defined recursively as $\tau^G = \{\sigma ^{G} : \exists p \in G((\sigma, p) \in \tau)\}$ and the generic
extension is defined as $V[G]=\{\tau^{G} : \tau \in V^{\mathbb{P}}\}$. 
 We recall that $V[G]$ is the smallest transitive
model of $\mathsf{ZFC}$ which has the same ordinals as $V$ and contains both $V$ and $G$. If $\mathbb{P}$ is our forcing notion and $G$ is a $\mathbb{P}$-generic filter  over $V$, we will abuse notation somewhat and use both $V^{\mathbb{P}}$ and $V[G]$ to denote the generic extension of $V$.
We state a part of L\'{e}vy--Solovay Theorem (cf. \cite[\textbf{Theorem 21.2}]{Jec2003}) in $\mathsf{ZFC}$. 

\begin{thm}
{\em 
Let $\kappa$ be an infinite cardinal, and let $\mathbb{P}$ be a forcing notion of size less than $\kappa$. Let G be a $\mathbb{P}$-generic filter over V. 
\begin{enumerate}
    \item If $\kappa$ is Ramsey in V, then $\kappa$  is Ramsey in $V[G]$.
    \item If $\kappa$ is measurable with a $\kappa$-complete ultrafilter $\mathcal{U}$ in V, then $\kappa$ is measurable with a $\kappa$-complete ultrafilter $\mathcal{U}_1=\{X\subseteq \kappa: X\in V[G],\text{and there is a } Y\in\mathcal{U} \text{ such that   }Y\subseteq X\}$ defined in $V[G]$ generated by $\mathcal{U}$ in $V[G]$.
\end{enumerate}
}
\end{thm}

\begin{proof}
(1) follows from \cite[\textbf{Theorem 21.2}]{Jec2003} and (2) follows from \cite[\textbf{Theorem 10}]{LS1967}.
\end{proof}

\subsection{Symmetric extension} Symmetric extensions are submodels of the generic extension containing the ground model, where $\mathsf{AC}$ can consistently fail. Let $\mathbb{P}$ be a forcing notion, $\mathcal{G}$ be a group of automorphisms of $\mathbb{P}$ and $\mathcal{F}$ be a normal filter of subgroups over $\mathcal{G}$.
We recall the following Symmetry Lemma from \cite{Jec2003}. 

\begin{lem}{\textbf{(The Symmetry Lemma; cf. \cite[\textbf{Lemma 14.37}]{Jec2003})}}
{\em Let $\mathbb{P}$ be a forcing notion,  $\varphi$ be a formula of the forcing language with $n$ free variables and let $\sigma_1,\sigma_2,...,\sigma_n$ be $\mathbb{P}$-names. If $a$ is an automorphism of $\mathbb{P}$, then $p\Vdash \varphi(\sigma_1,\sigma_2,...,\sigma_n) \Longleftrightarrow a(p)\Vdash \varphi(a(\sigma_1),a(\sigma_2),...,a(\sigma_n))$.
}
\end{lem}

For $\tau\in V^{\mathbb{P}}$, we denote  its symmetric group with respect to $\mathcal G$ by $sym^{\mathcal G}(\tau) =\{g\in \mathcal G : g\tau = \tau\}$ and say $\tau$ is {\em symmetric} with respect to $\mathcal{F}$ if $sym^{\mathcal G}(\tau)\in\mathcal F$. Let $HS^\mathcal F$ be the class of all hereditary symmetric names.
That is, recursively for $\tau\in V^{\mathbb{P}}$, 

\begin{center}
    $\tau\in HS^{\mathcal{F}}$ iff $\tau$ is symmetric with respect to $\mathcal{F}$, and for each $\sigma\in dom(\tau)$, $\sigma\in HS^{\mathcal{F}}$.
\end{center}

We define the symmetric extension of $V$ with respect to $\mathcal F$ as $V(G)^\mathcal F = \{\tau^G : \tau\in HS^\mathcal F\}$. For the sake of our convenience we omit the superscript $\mathcal{F}$ sometimes and call $V(G)^\mathcal F$ as $V(G)$, $HS^\mathcal F$ as $HS$, and $sym^\mathcal F(\tau)$ as $sym(\tau)$.

\begin{defn}{\textbf{(Symmetric system; cf. \cite[\textbf{Definition 2.1}]{HK2019})}} We say $\langle\mathbb{P},\mathcal{G},\mathcal{F}\rangle$ is a {\em symmetric system} if $\mathbb{P}$ is a forcing notion, $\mathcal{G}$ is  a group of automorphisms of $\mathbb{P}$, and $\mathcal{F}$ is a normal filter of subgroups over $\mathcal{G}$.
\end{defn}

\begin{defn}{\textbf{(Tenacious system;  \textbf{cf. \cite[\textbf{ Definition 4.6}]{Kar2019a}})}}  Let $\langle \mathbb{P},\mathcal{G},\mathcal{F}\rangle$ be a symmetric system. A condition $p\in \mathbb{P}$ is {\em $\mathcal{F}$-tenacious} if $\{\pi\in\mathcal{G} : \pi(p)= p\}\in \mathcal{F}$. We say $\mathbb{P}$ is {\em $\mathcal{F}$-tenacious} if there is a dense subset of $\mathcal{F}$-tenacious conditions. We say $\langle\mathbb{P},\mathcal{G},\mathcal{F}\rangle$ is a {\em tenacious system} if $\mathbb{P}$ is $\mathcal{F}$-tenacious.
\end{defn}

Karagila and Hayut proved that every symmetric system is equivalent to a tenacious system (cf. \cite[\textbf{Appendix A}]{Kar2019a}). Thus, it is natural to assume tenacity and work with tenacious system. We recall the following theorem which states that the symmetric extension $V(G)$ is a transitive model of $\mathsf{ZF}$.
\begin{thm}
{\textbf{(cf. \cite[\textbf{Lemma 15.51}]{Jec2003})}} {\em If $\langle\mathbb{P},\mathcal{G},\mathcal{F}\rangle$ is a symmetric system and G is a $\mathbb{P}$-generic filter over $V$, then $V(G)$ is a transitive model of $\mathsf{ZF}$ and $V\subseteq V(G)\subseteq V[G]$.}
\end{thm}

\subsection{Terminologies} We recall the terminologies like {\em Approximation Lemma}, {\em Approximation property}, and {\em ($\mathcal{G},\mathcal{I}$)-homogeneous forcing notion}, from \cite{Dim2011} and \cite{Ina2013}.
For $E\subseteq\mathbb{P}$, let us define its pointwise stabilizer group to be 
fix$_{\mathcal{G}}E = \{g\in\mathcal{G}:\forall p \in E (g(p)=p)\}$, 
i.e., it is the set of automorphisms which fix $E$ pointwise. We denote fix$_{\mathcal{G}}E$ by fix $E$ for the sake of convenience.

\begin{defn}{\textbf{($\mathcal G$-symmetry generator; \cite[\textbf{Definition 14}]{Ina2013})}}
Let $\mathbb{P}$ be a forcing notion and $\mathcal{G}$ be a group of automorphisms of $\mathbb{P}$. A subset $\mathcal{I}\subseteq \mathcal P(\mathbb{P})$ is called a $\mathcal G$-{\em symmetry generator} if it consists of up-sets, is closed under unions, and if for all $g\in \mathcal G$ and $E\in \mathcal{I}$, there is an $E'\in\mathcal{I}$ such that $g$(fix$E$)$g^{-1}\supseteq$ fix$E'$. 
\end{defn}

We can see that if $\mathcal{I}$ is a $\mathcal G$-symmetry generator, then the set $\{$fix$E: E\in\mathcal{I}\}$ generates a normal filter over $\mathcal G$ (cf. \cite[\textbf{Lemma 15}]{Ina2013}).

\begin{defn}
Let $\mathcal{I}$ be the $\mathcal G$-symmetry generator, we say $E\in \mathcal{I}$ {\em supports} a name $\sigma\in HS$ if fix$E\subseteq sym (\sigma$).  
\end{defn}

\begin{defn}{\textbf{(Projectable $\mathcal{G}$-symmetry generator; \cite[\textbf{Definition 1.25}]{Dim2011} $\&$ \cite[\textbf{Definition 17}]{Ina2013})}}
Let $\mathbb{P}$ be a forcing notion, $\mathcal{G}$ be a group of automorphisms of $\mathbb{P}$, and $\mathcal{I}$ be a $\mathcal{G}$-symmetry generator.
We say $\mathcal{I}$ is {\em projectable} for the pair ($\mathbb{P},\mathcal{G}$) if for every $p\in\mathbb{P}$ and every $E\in \mathcal{I}$, there is a $p^{*}\in E$ that is minimal  (with respect to the partial order) and unique such that $p^{*}\geq p$.
We call $p\restriction E = p^{*}$ the {\em projection} of $p$ to $E$. 
\end{defn}

For the rest of this section, let $\mathbb{P}$ be a forcing notion, $\mathcal{G}$ be a group of automorphisms of $\mathbb{P}$, and $\mathcal{I}$ be a projectable $\mathcal{G}$-symmetry generator for the pair ($\mathbb{P},\mathcal{G}$). 

\begin{defn}{\textbf{(Approximation property;  \cite[\textbf{Definition 18}]{Ina2013})}}
We say that the triple $\langle\mathbb{P},\mathcal{G},\mathcal{I}\rangle$ has the {\em approximation property} if for any formula $\varphi$ with $n$ free variables, and names $ \sigma_1,\sigma_2,...,\sigma_n\in HS$ all with support $E\in\mathcal{I}$, and for any $p\in \mathbb{P}$, $p\Vdash\varphi(\sigma_1,\sigma_2,...,\sigma_n)$ implies that $p\restriction E \Vdash\varphi(\sigma_1,\sigma_2,...,\sigma_n) $. 
\end{defn}

\begin{defn}{\textbf{($(\mathcal{G},\mathcal{I})$-homogeneous forcing notion; \cite[\textbf{Definition 1.26}]{Dim2011} $\&$ \cite[\textbf{Definition 19}]{Ina2013})}}
We say that $\mathbb{P}$ is {\em $(\mathcal{G},\mathcal{I})$-homogeneous} if for every $E\in \mathcal{I}$, every $p\in\mathbb{P}$, and every $q\in\mathbb{P}$ such that $q\leq p\restriction E$, there is an automorphism $a\in$ fix$E$ such that $a(p)\parallel q$.
\end{defn}

\begin{lem}
{\textbf{(\cite[\textbf{Lemma 1.27}]{Dim2011} $\&$ \cite[\textbf{Lemma 20}]{Ina2013})}}
{\em If $\mathbb{P}$ is $(\mathcal{G},\mathcal{I})$-homogeneous, then $\langle\mathbb{P},\mathcal G,\mathcal{I}\rangle$ has the approximation property.}
\end{lem}

We note that if $E \in \mathcal{I}$, then $E$ itself is a forcing notion (with the same top element as $\mathbb{P}$). So, we can say that set of pairs $\tau$ is an $E$-name iff $\tau$ is a relation and for every $(\sigma, p) \in\tau$, $\sigma$ is an $E$-name and $p \in E$.

\begin{lem}
{\textbf{(Approximation Lemma; \cite[\textbf{Lemma 1.29}]{Dim2011} $\&$ \cite[\textbf{Lemma 21}]{Ina2013})}}
{\em If the triple $\langle\mathbb{P},\mathcal G,\mathcal{I}\rangle$ has the approximation property then for all set of ordinals $X\in V(G)$, there exists an $E\in \mathcal{I}$ and an $E$-name for $X$. Thus, $X\in V[G\cap E]$.}
\end{lem}

\subsection{Homogeneity of forcing notions} We recall the definition of {\em weakly homogeneous} and
{\em cone homogeneous} forcing notion from \cite{DF2008}. 

\begin{defn}{\textbf{(Weakly homogeneous forcing notion; \cite[\textbf{Definition 2}]{DF2008})}}
We say a set forcing notion $\mathbb{P}$ is {\em weakly homogeneous} if and only if  for any $p,q\in \mathbb{P}$, there is an automorphism $a:\mathbb{P}\rightarrow\mathbb{P}$ such that $a(p)$ and $q$ are compatible.\footnote{The Levy collapse $Col(\lambda,<\kappa)$ is weakly homogeneous, given an infinite cardinal $\kappa$ and a regular cardinal $\lambda$.}
\end{defn}

\begin{defn}{\textbf{(Cone homogeneous forcing notion; \cite[\textbf{Definition 2}]{DF2008})}}
For $p\in\mathbb{P}$, let $Cone(p)$ denote $\{r\in\mathbb{P}:r\leq p\}$, the cone of conditions in $\mathbb{P}$ below $p$. We say a set forcing notion $\mathbb{P}$ is {\em cone homogeneous} if and only if for any $p, q\in \mathbb{P}$, there exist $p'\leq p$, $q'\leq q,$ and an isomorphism $\pi: Cone(p')\rightarrow Cone(q')$.
\end{defn}

If $\mathbb{P}$ is weakly homogeneous, then it is cone homogeneous too (cf. \cite[\textbf{Fact 1}]{DF2008}). Also, the finite support products of weakly (cone) homogeneous forcing notions are weakly (cone) homogeneous. A crucial feature of symmetric extensions using weakly (cone) homogeneous forcings are that they can be approximated by certain intermediate submodel where $\mathsf{AC}$ holds. 

\subsection{Failure of weak choice principles} 
We use $\mathsf{AC_{\kappa}}$ to denote the statement ``Every family of $\kappa$ non-empty sets admits a choice function''. We note that if $\kappa^{+}$ is singular, then $\mathsf{AC_{\kappa}}$ fails. This is due to the following well known fact.

\begin{fact}
{\em For all successor cardinal $\lambda$, $\mathsf{AC_{\kappa}}$ implies $cf(\lambda)>\kappa$.} 
\end{fact}

We sketch another way of refuting $\mathsf{AC_{\kappa}}$. 
For sets $A$ and $B$, we use $\mathsf{AC_{A}(B)}$ to denote the statement ``for each set $X$ of non-empty subsets of $B$, if there is an injection from $X$ to $A$ then there is a choice function for $X$". We recall \cite[\textbf{Lemmas 0.2, 0.3, 0.12}]{Dim2011}.
\begin{itemize}
\item Under $\mathsf{AC_{A}(B)}$, if there is a surjection  from $B$ to $A$, then there is an injection from $A$ to $B$ (cf. \cite[\textbf{Lemma 0.2}]{Dim2011}). 

\item For every infinite cardinal $\kappa$, there is a surjection from $\mathcal{P}(\kappa)$ onto $\kappa^{+}$ in $\mathsf{ZF}$ (cf. \cite[\textbf{Lemma 0.3}]{Dim2011}).

\item If $\kappa$ is measurable with a normal measure or $\kappa$ is weakly compact and $\alpha<\kappa$, then there is no injection $f:\kappa\rightarrow \mathcal{P}(\alpha)$ in $\mathsf{ZF}$ (cf. \cite[\textbf{Proposition 0.1}]{Bul1978}, \cite[\textbf{Lemma 0.12}]{Dim2011}). 
\end{itemize}

The following lemma states that if a successor cardinal $\kappa$ is either measurable with a normal measure or weakly compact then $\mathsf{AC_{\kappa}}$ fails, which is \cite[\textbf{Corollary 0.3}]{Bul1978}.

\begin{lem}
{\em Let $\kappa=\alpha^{+}$ be a successor cardinal. If $\kappa$ is measurable with a normal measure or weakly compact then $\mathsf{AC_{\alpha^{+}}(\mathcal{P}(\alpha))}$ fails}. 
\end{lem}

\begin{proof}
Let $\mathsf{AC_{\alpha^{+}}(\mathcal{P}(\alpha))}$ holds. We show $\kappa=\alpha^{+}$ is neither measurable with a normal measure nor weakly compact. In $\mathsf{ZF}$, there is a surjection from $\mathcal{P}(\alpha)$ onto $\alpha^{+}$. Now $\mathsf{AC_{\alpha^{+}}(\mathcal{P}(\alpha))}$ implies there is an injection $f'$ from $\alpha^{+}$ to $\mathcal{P}(\alpha)$ which states that $\kappa=\alpha^{+}$ is neither measurable with a normal measure nor weakly compact.  
\end{proof}

\section{Preserving Dependent Choice in symmetric extensions} 
Dependent Choice, denoted by $\mathsf{DC}$ or $\mathsf{DC_{\omega}}$, is a weaker version of $\mathsf{AC}$ which is strictly stronger\footnote{In Howard--Rubin's first model ($\mathcal{N}_{38}$ in \cite{HR1998}), $\mathsf{AC_{\omega}}$ holds but $\mathsf{DC_{\omega}}$ fails.} than the Countable Choice, denoted by $\mathsf{AC_{\omega}}$. This principle is strong enough to give the basis of analysis as it is equivalent to the Baire Category Theorem which is a fundamental theorem in functional analysis. Further, $\mathsf{DC}$ is equivalent to other important theorems like the countable version of the Downward L\"{o}weinheim--Skolem theorem and every tree of height $\omega$ without a maximum node has an infinite branch etc. On the other hand, $\mathsf{AC}$ has several controversial applications like the existence of a non-Lebesgue measurable set of real numbers, Banach--Tarski Paradox and the existence of a well-ordering of real numbers whereas $\mathsf{DC}$ does not have such counter-intuitive consequences. Thus it is desirable to preserve $\mathsf{DC}$ in symmetric extensions. 

For an infinite cardinal $\kappa$, we denote by $\mathsf{DC_{\kappa}}$ the principle of Dependent Choice for $\kappa$. This principle states that for every non-empty set $X$, if $R$ is a binary relation such that for each ordinal $\alpha<\kappa$, and each $f:\alpha\rightarrow X$ there is some $y\in X$ such that $f$ $R$ $y$, then there is $f:\kappa\rightarrow X$ such that for each $\alpha<\kappa$, $f\restriction\alpha$ $R$ $f(\alpha)$.
We denote the assertion $(\forall\lambda<\kappa)\mathsf{DC_{\lambda}}$ by $\mathsf{DC_{<\kappa}}$. We recall that $\mathsf{AC}$ is equivalent to $\forall\kappa(\mathsf{DC_{\kappa}})$ and $\mathsf{DC_{\kappa}}$ implies $\mathsf{AC_{\kappa}}$.  We refer the reader to \cite[\textbf{Chapter 8}]{Jec1973}, for details concerning $\mathsf{DC_{\kappa}}$ and related choice principles.

We recall the definitions of a forcing notion with the {\em $\kappa$-chain condition} ($\kappa$-c.c.) and
of forcing notions that are {\em $\kappa$-closed} and {\em $\kappa$-distributive} from \cite[\textbf{Definition 5.8}]{Cum2010}. We also recall the definitions of forcing notions that are {\em $<\kappa$-strategically closed}, {\em $\kappa$-strategically closed}, and {\em $(\kappa+1)$-strategically closed} from \cite[\textbf{Definition 5.14, Definition 5.15}]{Cum2010}.
Monro \cite[\textbf{Corollary 5.11}]{Mon1983} proved that $\mathsf{DC_{\kappa}}$ is
not preserved by generic extensions for any infinite cardinal $\kappa$.

Karagila \cite[\textbf{Lemma 1}]{Kar2014} proved that if $\mathbb{P}$ is $\kappa$-closed and $\mathcal{F}$ is $\kappa$-complete then $\mathsf{DC_{<\kappa}}$ is preserved in the symmetric extension in terms of symmetric system $\langle \mathbb{P},\mathcal{G},\mathcal{F}\rangle$.
The author and Karagila both observed independently that ``$\mathbb{P}$ is $\kappa$-closed'' can be replaced by ``$\mathbb{P}$ has the $\kappa$-c.c.'' in  \cite[\textbf{Lemma 1}]{Kar2014}. The author noticed this observation combining the role of $\kappa$-c.c. forcing notions from \cite[\textbf{Lemma 2.2}]{Apt2001} and the role of $\kappa$-completeness of $\mathcal{F}$ from \cite[\textbf{Lemma 1}]{Kar2014}. 

The idea was the following. If $\mathbb{P}$ has the $\kappa$-c.c., then any antichain is of size less than $\kappa$. So by Zorn's Lemma in the ground model, there is a maximal antichain of conditions $\mathcal{A}=\{p_{\alpha}: \alpha<\gamma<\kappa\}$ extending $p$ such that for all $\alpha<\gamma$, $p_{\alpha}\Vdash \dot{f}(\hat{\alpha})=\dot{t}_{\alpha}$ where $\dot{t}_{\alpha}\in HS$. Then we can follow the proof of \cite[\textbf{Lemma 1}]{Kar2014} to finish the proof.

In a private conversation with Karagila, the author came to know that they independently observed the same fact. We note that there was a gap in the above observation. Specifically, the author was not aware of the fact that every symmetric system is equivalent to a tenacious system. Karagila fixed this gap. In particular, in \cite[\textbf{Lemma 3.3}]{Kar2019}, Karagila wrote that the natural assumption that $\langle\mathbb{P},\mathcal{G},\mathcal{F}\rangle$ is a tenacious system is also required in the proof. 

\begin{lem}{(\textbf{Karagila}; \cite[\textbf{Lemma 3.3}]{Kar2019})}
{\em Let $V$ be a model of $\mathsf{ZFC}$. If $\mathbb{P}$ has the $\kappa$-c.c. and $\mathcal{F}$ is $\kappa$-complete, then $\mathsf{DC_{<\kappa}}$ is preserved in the symmetric extension of $V$ with respect to the symmetric system $\langle \mathbb{P},\mathcal{G},\mathcal{F}\rangle$.}
\end{lem}

We slightly generalize \cite[\textbf{Lemma 1}]{Kar2014} and prove \textbf{Observation 1.1}.

\begin{lem}{\textbf{(Observation 1.1)}}
{\em Let $V$ be a model of $\mathsf{ZFC}$. If $\mathbb{P}$ is $\kappa$-distributive and $\mathcal{F}$ is $\kappa$-complete, then $\mathsf{DC_{<\kappa}}$ is preserved in the symmetric extension of $V$ with respect to the symmetric system $\langle \mathbb{P},\mathcal{G},\mathcal{F}\rangle$.}
\end{lem}

\begin{proof}
Let $G$ be a $\mathbb{P}$-generic filter over $V$. Let $\delta<\kappa$. We show that $\mathsf{DC_{\delta}}$ holds in $V(G)$. Let $X$ and $R$ be elements of $V(G)$ as in the assumptions of $\mathsf{DC_{\delta}}$. Since $\mathsf{AC}$ is equivalent to $\forall\kappa$($\mathsf{DC_{\kappa}}$) and $V[G]$ is a model of $\mathsf{AC}$, using $\forall\kappa$($\mathsf{DC_{\kappa}}$) in $V[G]$, we can find an $f:\delta\rightarrow X$ in $V[G]$ such that $f\restriction \alpha$ $R$ $f(\alpha)$ for all $\alpha < \delta$. We show that this $f$ is in $V(G)$.
Let $p_{0}\Vdash \dot{f}$ is a function whose domain is $\delta$ and range is $X$ which is a subset of $V(G)$. 
For each $\alpha<\delta$,  $D_{\alpha}=\{p\leq p_{0}: (\exists x\in X) p\Vdash \dot{f}(\check{\alpha})=\dot{x}$ where $\dot{x}\in HS \}$ is open dense below $p_{0}$. Consequently by $\kappa$-distributivity of $\mathbb{P}$, $D=\bigcap_{\alpha<\delta} D_{\alpha}$ is dense below $p_{0}$. So, there is some $p\in D\cap G$. We can see that for each $\alpha<\delta$, there is an $x_{\alpha}$ such that $p\Vdash \dot{f}(\check{\alpha})=\dot{x}_{\alpha}$ where $\dot{x}_{\alpha}\in HS$. Define $\dot{g}=\{\langle\check{\alpha},\dot{x}_{\alpha}\rangle:\alpha<\delta\}$. Now, since each $\dot{x}_{\alpha}\in HS$, $sym(\dot{x}_{\alpha})\in \mathcal{F}$. By $\kappa$-completeness of $\mathcal{F}$,
$H=\bigcap_{\alpha<\delta}sym(\dot{x}_{\alpha})\in \mathcal{F}$. Next, since $H$ is a subgroup of $sym{(\dot{g})}$
and $\mathcal{F}$ is a filter, $\dot{g}\in HS$. We can see that $p\Vdash \dot{g}=\dot{f}$. Thus, there is a dense open set of conditions $q\leq p_{0}$, such that for some $\dot{g}\in HS$, $q\Vdash\dot{g}=\dot{f}$. By genericity, $\dot{f}^{G}=f\in V(G)$.
\end{proof}

\begin{remark}If $\kappa$ is either a supercompact cardinal or a strongly compact cardinal and $\lambda>\kappa$ is a regular cardinal, there are certain forcing notions like supercompact Prikry forcing \cite{Apt1985} and strongly compact Prikry forcing \cite{AH1991} which are known to be non-$\kappa$-distributive, but still can preserve $\mathsf{DC_{\kappa}}$ in the symmetric extension based on such forcings.
In particular, Apter communicated to us that, assuming the consistency of a $2^{\lambda}$-supercompact cardinal $\kappa$ and a regular cardinal $\lambda>\kappa$, Kofkoulis proved in \cite{Kof1990}, that in a symmetric extension based on supercompact Prikry forcing, $\mathsf{DC_{\kappa}}$ was preserved. In particular, $\mathsf{DC_{\kappa}}$ holds in the symmetric inner model constructed in \cite[\textbf{Theorem 1}]{Apt1985}. Further applying the methods of Kofkoulis, assuming the consistency of a $\lambda$-strongly compact cardinal $\kappa$ and a measurable cardinal $\lambda>\kappa$, a symmetric extension based on strongly compact Prikry forcing was constructed in \cite{AH1991} where $\kappa$ became a singular cardinal of cofinality $\omega$, $\kappa^{+}$ remained a measurable cardinal and $\mathsf{DC_{\kappa}}$ was preserved. We can also find another exhibition of Kofkoulis's method with certain modifications in \cite{AM1995}.
\end{remark}

Next, we prove \textbf{Observation 1.2}.

\begin{lem}{\textbf{(Observation 1.2)}}
{\em Let $\delta<\kappa$ and $V$ be a model of $\mathsf{ZF+DC_{\delta}}$. If $\mathbb{P}$ is ($\delta+1$)-strategically closed and $\mathcal{F}$ is $\kappa$-complete, then $\mathsf{DC_{\delta}}$ is preserved in the symmetric extension of $V$ with respect to the symmetric system $\langle \mathbb{P},\mathcal{G},\mathcal{F}\rangle$. }
\end{lem}

\begin{proof}
Let $\delta<\kappa$. Let $G$ be a $\mathbb{P}$-generic filter over $V$. By \cite[\textbf{Theorem 2.2}]{GJ2014}, $\mathsf{DC_{\delta}}$ is preserved in $V[G]$ since $\mathbb{P}$ is ($\delta+1$)-strategically closed. We show that $\mathsf{DC_{\delta}}$ holds in $V(G)$. Let $X$ and $R$ be elements of $V(G)$ as in the assumptions of $\mathsf{DC_{\delta}}$. Since $\mathsf{DC_{\delta}}$ is preserved in $V[G]$, we can find an $f:\delta\rightarrow X$ in $V[G]$ such that $f\restriction \alpha$ $R$ $f(\alpha)$ for all $\alpha < \delta$. We show that this $f$ is in $V(G)$.
Let $p\Vdash \dot{f}$ is a function whose domain is $\delta$ and its range is a subset of $V(G)$.  
Consider a game of length $\delta+1$, between two players I and II who play at odd stages and even stages respectively such that initially II chooses a trivial condition and I chooses a condition $p_{0}$ extending $p$ such that $p_{0}\Vdash \dot{f}(\check{0})=\dot{t}_{0}$ where $\dot{t}_{0}$ is in $HS$, and at non-limit even stages 2$\alpha>0$, II chooses a condition $p_{\alpha}$ extending the condition of the previous stage such that $p_{\alpha}\Vdash \dot{f}(\check{\alpha})=\dot{t}_{\alpha}$ where $\dot{t}_{\alpha}$ is in $HS$. By ($\delta+1$)-strategic closure of $\mathbb{P}$, II has winning strategy. Thus, we can we can extend
$p$ to $p_{0} \geq p_{1} \geq$ ··· $\geq p_{\alpha}$ $\geq$ ···$\geq p_{\delta}$ such that $p_{\alpha}\Vdash \dot{f}(\check{\alpha})=\dot{t}_{\alpha}$ where $\dot{t}_{\alpha}$ is in $HS$ for each $\alpha<\delta$. It is enough to show that $\dot{f}=\{\langle \check{\beta},\dot{t}_{\beta}\rangle: \beta<\delta\}$ is in $HS$ which follows using $\kappa$-completeness of $\mathcal{F}$ as done in \cite[\textbf{Lemma 1}]{Kar2014}.
\end{proof}

\begin{remark}
We note that we are using the  definition of a $(\delta+1)$-strategically closed  forcing notion
from \cite[\textbf{Definition 5.15}]{Cum2010}
which is different from the definition used in \cite{GJ2014}. In particular, in our case a forcing  notion is $\delta$-strategically closed if in the two person game in which the players construct a descending sequence $\langle p_{\alpha} : \alpha < \delta\rangle$, where player I plays odd stages and player II plays even and limit stages
(choosing the trivial condition at stage 0), player II has a strategy which ensures the game
can always be continued; a forcing  notion is  ($\delta+1$)-strategically closed if the corresponding game has length $\delta+1$. 
Whereas Gitman and Johnstone defined that a forcing  notion is $\leq\delta$-strategically closed if in the game of ordinal length $\delta+1$ in which two players alternatively select conditions from it to construct a descending
$\delta+1$-sequence with the second player playing at even and limit stages, the second player has a
strategy that allows her to always continue playing (cf. \cite[the paragraph before \textbf{Theorem 2.2}]{GJ2014}).
Thus in our case if $V$ is a model of $\mathsf{ZF+DC_{\delta}}$ and $\mathbb{P}$ is ($\delta+1$)-strategically closed then $\mathsf{DC_{\delta}}$ is preserved in $V[G]$ by \cite[\textbf{Theorem 2.2}]{GJ2014}.
\end{remark}

\begin{remark}
Let $V$ be a model of $\mathsf{ZF +DC_{\kappa}}$. Suppose $\mathbb{P}$ is well-orderable of order type at most $\kappa$ and has the $\kappa$-c.c.  property. We remark that if $\mathcal{F}$ is $\kappa$-complete, then $\mathsf{DC_{<\kappa}}$ is preserved in the symmetric extension of $V$ in terms of the symmetric system $\langle \mathbb{P},\mathcal{G},\mathcal{F}\rangle$.
Let $G$ be a $\mathbb{P}$-generic filter over $V$.
By \cite[\textbf{Theorem 2.1}]{GJ2014}, $\mathsf{DC_{\kappa}}$ is preserved in $V[G]$. The rest follows from the proof of \cite[\textbf{Lemma 3.3}]{Kar2019}.
\end{remark}

\begin{question}
{\em Let V be a model of $\mathsf{ZF +DC_{\kappa}}$. Suppose that $\mathbb{P}$ is $\kappa$-distributive. Can we preserve $\mathsf{DC_{\kappa}}$ in every forcing extension $V[G]$ by $\mathbb{P}$?}
\end{question}

If the answer is in the affirmative, then we can say that if $V$ is a model of $\mathsf{ZF +DC_{\kappa}}$, $\mathbb{P}$ is $\kappa$-distributive, and $\mathcal{F}$ is $\kappa$-complete, then $\mathsf{DC_{<\kappa}}$ is preserved in the symmetric extension in terms of the symmetric system $\langle \mathbb{P},\mathcal{G},\mathcal{F}\rangle$ following \textbf{Lemma 3.2}.

\subsection{Number of normal measures a successor cardinal can carry and Dependent Choice} 

Takeuti \cite{Tak1970} and  Jech \cite{Jec1968} independently proved that
if we assume the consistency of ``$\mathsf{ZFC}$ + there is a measurable cardinal" then the theory ``$\mathsf{ZF+DC}$ + $\aleph_{1}$ is a measurable cardinal'' is consistent. 
Dimitriou \cite[\textbf{section 1.33}]{Dim2011} modified Jech's construction and proved that if we assume the consistency of ``$\mathsf{ZFC}$ + there is a measurable cardinal $\kappa$ and $\gamma<\kappa$ is a regular cardinal'' then the theory ``$\mathsf{ZF}$ + the cardinality of $\gamma$ is preserved + $\gamma^{+}$ is a measurable cardinal'' is consistent. Apter, Dimitriou, and Koepke \cite{ADK2014} constructed symmetric models in which for an arbitrary ordinal $\rho$, $\aleph_{\rho+1}$ can be the least measurable as well as the least regular uncountable cardinal. 
Bilinsky and Gitik \cite{BG2012} proved that if we assume the consistency of ``$\mathsf{ZFC+GCH}$ + there is a measurable cardinal $\kappa$'' then we can obtain a symmetric extension where $\kappa$ is a measurable cardinal without a normal measure. 
Assuming the consistency of {``$\mathsf{ZFC+GCH}$ + there is a measurable cardinal"}, we can construct models of $\mathsf{ZF+DC}$ where successor of regular cardinals like $\aleph_1$, $\aleph_{2}$, $\aleph_{\omega+2}$, as well as $\aleph_{\omega_{1}+2}$, can carry an arbitrary (non-zero) number of normal measures. 

Friedman--Magidor \cite[\textbf{Theorem 1}]{FM2009} proved that a measurable cardinal can be forced to carry arbitrary number of normal measures in $\mathsf{ZFC}$. 
\begin{lem}{(\textbf{Friedman and Magidor}; \cite[\textbf{Theorem 1}]{FM2009})} {\em Assume $\mathsf{GCH}$. Suppose that $\kappa$ is a measurable cardinal and let $\alpha$ be a cardinal at most $\kappa^{++}$. Then in a cofinality-preserving forcing extension, $\kappa$ carries exactly $\alpha$ normal measures.}\end{lem}

We recall the definition of a {\em symmetric collapse} from \cite{HK2019}.

\begin{defn}{(\textbf{Symmetric Collapse}; \cite[\textbf{Definition 4.1}]{HK2019})} Let $\kappa\leq\lambda$ be two infinite cardinals. The {\em symmetric collapse} is the symmetric system $\langle \mathbb{P},\mathcal{G},\mathcal{F}\rangle$ defined as follows.
\begin{itemize}
\item $\mathbb{P}=Col(\kappa,<\lambda)$, so a condition in $\mathbb{P}$ is a partial function $p$ with domain
$\{\langle \alpha, \beta\rangle : \kappa < \alpha < \lambda, \beta < \kappa\}$ such that $p(\alpha,\beta) < \alpha$ for all $\alpha$ and $\beta$,
supp($p$) = $\{\alpha < \lambda : \exists\beta, \langle \alpha, \beta\rangle \in$ dom $p\}$ is bounded below $\lambda$ and $\vert p\vert < \kappa$.
\item $\mathcal{G}$ is the group of automorphisms $\pi$ such that there is a sequence of permutations $\overrightarrow{\pi}=\langle\pi_{\alpha}:\kappa<\alpha<\lambda \rangle$ such that $\pi_{\alpha}$ is a permutation of $\alpha$ satisfying $\pi p(\alpha,\beta)=\pi_{\alpha}(p(\alpha,\beta))$.
\item $\mathcal{F}$ is the normal filter of subgroups generated by fix$(E)$ for bounded $E\subseteq\lambda$, where fix$(E)$ is the group $\{\pi: \forall\alpha\in E, \pi p(\alpha,\beta)=p(\alpha,\beta)\}$.\end{itemize}
\end{defn}

\begin{lem}{\em Let $\kappa\leq\lambda$ be two infinite cardinals such that $cf(\lambda)\geq\kappa$ and $\langle \mathbb{P},\mathcal{G},\mathcal{F}\rangle$ is the symmetric collapse where $\mathbb{P}=Col(\kappa,<\lambda)$. Then, $\mathcal{F}$ is $\kappa$-complete.}\end{lem}

\begin{proof}Fix  $\gamma<\kappa$ and let, for each $\beta<\gamma$, $K_{\beta}\in \mathcal{F}$. There must be bounded $E_{\beta}\subseteq\lambda$ for each $\beta<\gamma$ such that fix$E_{\beta}\subseteq K_{\beta}$. Next, fix($\bigcup_{\beta<\gamma}E_{\beta})\subseteq\bigcap_{\beta<\gamma}$ fix$E_{\beta}\subseteq \bigcap_{\beta<\gamma} K_{\beta}$. Since $cf(\lambda)\geq\kappa$, $\bigcup_{\beta<\gamma}E_{\beta}$ is a bounded subset of $\lambda$. Consequently, $\bigcap_{\beta<\gamma} K_{\beta}\in \mathcal{F}$.\end{proof}

We observe that after a symmetric collapse, the successor of a regular cardinal can be a measurable cardinal carrying an arbitrary (non-zero) number of normal measures assuming the consistency of a measurable cardinal. Further we can preserve Dependent Choice in certain cases.

\begin{thm}
{\em Let $V$ be a model of $\mathsf{ZFC+GCH}$ with a measurable cardinal $\kappa$. Let $\lambda$ be any non-zero cardinal at most $\kappa^{++}$ and let $\eta\leq\kappa$ be regular. Then, there is a symmetric extension where $\kappa=\eta^{+}$ is a measurable cardinal carrying $\lambda$ normal measures. Moreover, $\mathsf{AC_{\kappa}}$ fails and $\mathsf{DC_{<\eta}}$ holds\footnote{If we assume $\eta>\omega$.} in the symmetric model}.\end{thm}

\begin{proof} Applying \textbf{Lemma 3.8}, we obtain a cofinality-preserving forcing extension $V'$ of $V$ where $\kappa$ is a measurable cardinal with $\lambda$ many normal measures. Let $V'(G)$ be the symmetric extension of $V'$ obtained by the symmetric collapse $\langle \mathbb{P},\mathcal{G},\mathcal{F}\rangle$ where $\mathbb{P}=Col(\eta,<\kappa)$ and $G$ is a $\mathbb{P}$-generic filter over $V'$. In $V'(G)$, $\kappa=\eta^{+}$. We can also have the following in $V'(G)$.
\begin{itemize}
\item By \cite[\textbf{Lemmas 2.4, 2.5}]{Apt2001}, $\kappa$ remains a measurable cardinal with $\lambda$ many normal measures. \item Since $\kappa$ is a successor as well as a measurable cardinal, $\mathsf{AC_{\kappa}}$ fails using \textbf{Lemma 2.20}.
\item We note that $\mathcal{F}$ is $\eta$-complete by \textbf{Lemma 3.10}.
Since $\mathbb{P}$ is $\eta$-closed, $\mathsf{DC_{<\eta}}$  holds using \cite[\textbf{Lemma 1}]{Kar2014}. \end{itemize}\end{proof}

\begin{remark}
The referee pointed out that $\mathsf{DC_{<\kappa}}$ is preserved in $V'(G)$. Assuming that $\lambda$ is regular, the proof of \textbf{Lemma 3.10} gives that $\mathcal{F}$ is $\lambda$-complete. Consequently, since $\kappa$ is a regular cardinal in $V'$, $\mathcal{F}$ is $\kappa$-complete. Since $\mathbb{P}$ has the $\kappa$-c.c., by \textbf{Lemma 3.1}, $\mathsf{DC_{<\kappa}}$ is preserved in $V'(G)$.   
\end{remark}

\begin{remark}
In \cite[\textbf{Theorem 1}]{Apt2001},
starting with a model of $``\mathsf{ZFC}+ \mathsf{GCH}+ o(\kappa)=\delta^{*}"$ for $\delta^{*}\leq\kappa^{+}$ any finite or infinite cardinal, Apter constructed a model of $\mathsf{ZF}+\mathsf{DC_{<\kappa}}$ where $\kappa$ carries exactly $\delta^{*}$ normal measures and $2^{\delta} = \delta^{++}$ on a set having measure one with respect to every
normal measure over $\kappa$. We observe that we can obtain the result of  \cite[\textbf{Theorem 1}]{Apt2001} starting from just one measurable cardinal $\kappa$ if we use \textbf{Lemma 3.8} instead of passing to an inner model of Mitchell from \cite{Mit1974} as done in the proof of  \cite[\textbf{Theorem 1}]{Apt2001}. In particular, we can prove the following.
\end{remark}

\begin{corr}{\textbf{(of \cite[\textbf{Theorem 1}]{Apt2001})}}
{\em Let $V$ be a model of $\mathsf{ZFC + GCH}$ with a measurable cardinal $\kappa$ and let $\lambda$ be a cardinal at most $\kappa^{++}$. Then there is a model of $\mathsf{ZF+\mathsf{DC_{<\kappa}}}$ where $\kappa$ is a measurable cardinal carrying $\lambda$ many normal measures $\langle \mathcal{U}_{\alpha}^{*}: \alpha<\lambda\rangle$. Moreover, we have $2^{\delta} = \delta^{++}$ on a set having measure one with respect to  any of
the measures $\mathcal{U}_{\alpha}^{*}$.}
\end{corr}

\section{The Proof of Theorem 1.5} 
In this section we prove \textbf{Theorem 1.5}. 

\begin{proof}{\textbf{(of Theorem 1.5)}}
Firstly, we give a description of the symmetric extension constructed in \cite[\textbf{Theorem 4.1}]{Kar2019} as follows.

\begin{enumerate}
\item \textbf{The ground model ($V$).}
At the beginning of the proof of \cite[\textbf{Theorem 3}]{Apt2012}, from the given requirements, Apter constructed a model $V$ where
there is an enumeration $\langle\kappa_{i} : i <\kappa\rangle$ of $C\cup \{\omega\}$ where  $C\subseteq\kappa$ is a club of inaccessible and limit cardinals below a supercompact cardinal $\kappa$ such that $2^{\kappa_{i}} = \kappa_{i}^{++}$ holds. We consider $V$ to be our ground model. For reader's convenience we recall the steps from the proof of \cite[\textbf{Theorem 3}]{Apt2012} as follows.

\begin{itemize}
    \item Let $V$ be a model of $\mathsf{ZFC + GCH}$ with a supercompact cardinal $\kappa$. 
    \item Let $\mathbb{Q}_{1}$ be Laver’s partial ordering of \cite{Lav1978} which makes $\kappa$’s supercompactness indestructible under $\kappa$-directed closed forcing. Since $\mathbb{Q}_{1}$ may be defined so that $\vert\mathbb{Q}_{1}\vert=\kappa$, we have $V^{\mathbb{Q}_{1}* \dot{Add}(\kappa,\kappa^{++})}$= $V_{2}$ is a model of ``$\mathsf{ZFC}$ + $\kappa$ is supercompact + $2^{\kappa} = \kappa^{++}$ + $2^{\delta} = \delta^{+}$ for every cardinal $\delta \geq \kappa^{+}$".
    \item Let $\mathbb{Q}_{3}\in V_{2}$ be the Radin forcing defined over $\kappa$. Taking a suitable measure sequence will enable one to preserve the supercompactness of $\kappa$ (cf. \cite{Git2010}). Consequently, $V_{2}^{\mathbb{Q}_{3}}= \overline{V}$ is a model of ``$\mathsf{ZFC}$ + $\kappa$ is supercompact + $2^{\kappa} = \kappa^{++}$ + $2^{\delta} = \delta^{+}$ for every cardinal $\delta \geq \kappa^{+}$ + There is a club $C \subseteq \kappa$ composed of inaccessible cardinals and their limits with $2^{\delta} = 2^{\delta^{+}}=\delta^{++}$ for every $\delta\in C$".
    \item For the sake of convenience we consider the ground model to be $\overline{V} = V$. Let $\langle \kappa_{i}: i<\kappa\rangle\in V$ be the continuous,
    increasing enumeration of $C\cup \{\omega\}$.
\end{itemize}

\item \textbf{Defining the symmetric system $\langle\mathbb{P},\mathcal{G},\mathcal{F}\rangle$.}  
\begin{itemize}
\item Let $\mathbb{P}$ be the Easton support product of  $\mathbb{P}_{\alpha}=Col(\kappa_{\alpha}^{++}, <\kappa_{\alpha+1})$ where $\alpha<\kappa$.
\item Let $\mathcal{G}$ be the Easton support product of the automorphism groups of each $\mathbb{P}_{\alpha}$.
\item Let $\mathcal{F}$ be the filter generated by the groups of the form fix$(\alpha)$ for $\alpha<\kappa$, where 

fix$(\alpha)=\{\pi\in \prod _{\beta<\kappa}Aut(\mathbb{P}_{\beta}): \pi\restriction\alpha = id\}$.
\end{itemize}

\item \textbf{Defining the symmetric extension of $V$.} Let $G$ be a $\mathbb{P}$-generic filter over $V$. We consider the symmetric extension $V(G)^{\mathcal{F}}$ with respect to  the symmetric system $\langle \mathbb{P},\mathcal{G},\mathcal{F}\rangle$ defined above in (2) and denote it by $\mathcal{N}_{1}$ for the sake of our convenience.
\end{enumerate}

Since $\mathcal{F}$ is $\kappa$-complete, and $\mathbb{P}$ has the $\kappa$-c.c., $\mathsf{DC_{<\kappa}}$ is preserved in $\mathcal{N}_{1}$ by \textbf{Lemma 3.1} (cf. \cite[\textbf{section 4.1}]{Kar2019}). Since each $\mathbb{P}_{\alpha}$ is weakly homogeneous, the following lemma holds as a corollary of \textbf{Lemma 2.16}.

\begin{lem}
{\em If $A\in \mathcal{N}_{1}$ is a set of ordinals, then $A\in V[G\restriction\alpha]$ for some $\alpha<\kappa$.}
\end{lem}

We apply \textbf{Lemma 4.1} to prove that $\kappa$ remains supercompact in $\mathcal{N}_{1}$ following the methods of \cite{Ina2013}. Inamder \cite{Ina2013} proved that if we assume the consistency of ``$\mathsf{ZFC}$ + there is a supercompact cardinal $\kappa$, and $\gamma<\kappa$ is a regular cardinal'' then the theory ``$\mathsf{ZF}$ + the cardinality of $\gamma$ is preserved + $\gamma^{+}$ is a supercompact cardinal'' is consistent.
We recall the relevant lemmas and incorporate the arguments from \cite{Ina2013} in order to show that $\kappa$ remains supercompact in $\mathcal{N}_{1}$. 

\begin{lem}{(cf. \cite[\textbf{Lemma 26}]{Ina2013})} {\em Let $\kappa$ be a regular cardinal, let $\gamma\geq \kappa$, and let $\mathbb{P}$ be a forcing  notion of size less than $\kappa$. Then for every $C\in \mathcal{P}_{\kappa}(\gamma)^{V[G]}$, there is a $D\in \mathcal{P}_{\kappa}(\gamma)^{V}$ such that in $V[G]$, $C\subseteq D$.}
\end{lem}

\begin{lem}{(\textbf{L\'{e}vy--Solovay Lemma; \cite[\textbf{Lemma 27}]{Ina2013}})} {\em In $V$, let $\kappa$ be a regular cardinal, $D$ be a set and $\mathcal{U}$ be a $\kappa$-complete ultrafilter on $D$. Let $\mathbb{P}$ be a forcing  notion of size less than $\kappa$ and $G$ be a $\mathbb{P}$-generic filter over $V$. Suppose $V[G]\models f:D\rightarrow V$. Then there is an $S\in\mathcal{U}$ and a $g:S\rightarrow V$ in $V$ such that $V[G]\models f\restriction S=g$.}
\end{lem}

Applying \textbf{Lemma 4.1} and \textbf{Lemma 4.3} we obtain the following lemma, which is similar to \cite[\textbf{Lemma 33}]{Ina2013}.

\begin{lem}
{\em Let $D$ be a set and $\mathcal{U}$ be a $\kappa$-complete ultrafilter on $D$ in $V$. Suppose $\mathcal{N}_{1}\models f:D\rightarrow V$. Then there is an $S\in \mathcal{U}$ and a $g:S\rightarrow V$ in $V$ such that $\mathcal{N}_{1}\models f\restriction S=g$.} 
\end{lem}

\begin{proof}
By \textbf{Lemma 4.1}, for some $\alpha<\kappa$ we get $f\in V[G \restriction \alpha]$. Now we can say $G \restriction\alpha$ is $\mathbb{P}'$-generic over $V$ where $\vert\mathbb{P}'\vert<\kappa$. By \textbf{Lemma 4.3}, we obtain an $S\in\mathcal{U}$ and a $g:S\rightarrow V$ in $V$ such that $V[G\restriction\alpha]\models f\restriction S=g$. So, $\mathcal{N}_{1}\models f\restriction S=g$.
\end{proof}

Similarly \cite[\textbf{Lemma 34}]{Ina2013}, we obtain the following lemma by applying \textbf{Lemma 4.4}.

\begin{lem}
{\em In $V$, let $D$ be a set and $\mathcal{U}$ be a $\kappa$-complete ultrafilter on $D$. Let $\mathcal{W}$ be the filter on $D$ generated by $\mathcal{U}$ in $\mathcal{N}_{1}$. Then $\mathcal{W}$ is a $\kappa$-complete ultrafilter.}
\end{lem}

We follow the proof of \cite[\textbf{Theorem 35}]{Ina2013} and refer the reader to \cite{Ina2013} for further details.

\begin{lem}
{\em In $\mathcal{N}_{1}$, $\kappa$ is supercompact.}
\end{lem}

\begin{proof}
Let $\gamma\geq\kappa$ be arbitrary.
Since $\kappa$ is supercompact in $V$, there is a normal measure $\mathcal{U}$ on $\mathcal{P}_{\kappa}(\gamma)$ in $V$. Let $\mathcal{V}$ be the $\kappa$-complete measure it generates on $\mathcal{P}_{\kappa}(\gamma)^{V}$ in $\mathcal{N}_{1}$. Let $\mathcal{W}$ be the filter generated by $\mathcal{V}$ on $\mathcal{P}_{\kappa}(\gamma)$ in $\mathcal{N}_{1}$. Since $\mathcal{W}$ is generated by a $\kappa$-complete ultrafilter on $\mathcal{P}_{\kappa}(\gamma)^{V}\subseteq \mathcal{P}_{\kappa}(\gamma)$, $\mathcal{W}$ is a $\kappa$-complete ultrafilter by \textbf{Lemma 4.5}.

\textbf{Fineness}: Let $X\in \mathcal{P}_{\kappa} (\gamma)^ {\mathcal{N}_{1}}$. By \textbf{Lemma 4.1}, for some $\alpha<\kappa$ we have $X\in V[G\restriction \alpha]$. Since $\kappa$ is not collapsed while going from $V$ to $V[G\restriction \alpha]$,
$X\in \mathcal{P}_{\kappa}(\gamma)^ {V[G \restriction\alpha]}$. 
By \textbf{Lemma 4.2} (and following the arguments in the proof of \cite[\textbf{Theorem 35}]{Ina2013}), $\hat{X}\in \mathcal{V}'$, where $\mathcal{V}'$ is the fine measure that $\mathcal{U}$ generates on $\mathcal{P}_{\kappa}(\gamma)^{V[G\restriction \alpha]}$. Now $\mathcal{U}\subseteq \mathcal{V}'\subseteq \mathcal{W}$
since $\mathcal{P}_{\kappa}(\gamma)^{V[G\restriction \alpha]}\subseteq \mathcal{P}_{\kappa}(\gamma)^{\mathcal{N}_{1}}$. Consequently $\mathcal{W}$ is fine.

\textbf{Choice function}: Let $\mathcal{N}_{1}\models f:\mathcal{P}_{\kappa}(\gamma)\rightarrow\gamma$ and $\mathcal{N}_{1}\models\forall X \in \mathcal{P}_{\kappa}(\gamma)(f(X)\in X)$. By \textbf{Lemma 4.1}, for some $\alpha<\kappa$ we get $h=f\restriction\mathcal{P}_{\kappa}(\gamma)^V \in V[G \restriction\alpha]$. By \textbf{Lemma 4.3}, we get $Y\in\mathcal{U}$ and $(g:Y\rightarrow\gamma)^V$ such that $V[G\restriction\alpha]\models h\restriction Y = g$. Now by normality of $\mathcal{U}$ in $V$ we get a set $x$ in $\mathcal{U}$ such that $g$ is constant on $x$, and so $h$ is constant on a set in $\mathcal{U}$. Hence,  we obtain a set $y$ in $\mathcal{W}$ such that $f$ is constant on $y$.
\end{proof}

\begin{lem}
{\em In $\mathcal{N}_{1}$, $\mathsf{AC_{\kappa}}$ fails.}
\end{lem}

\begin{proof}
Since the cardinality of $\kappa_{\alpha}^{++}$ is preserved in $\mathcal{N}_{1}$ for $\alpha<\kappa$, we can define in $\mathcal{N}_{1}$ the set $X_{\alpha}=\{x\subseteq \kappa_{\alpha}^{++} : x$ codes a well ordering of $(\kappa_{\alpha}^{+++})^{V}$ of order type $\kappa_{\alpha}^{++}\}$. We claim that $\langle X_{\alpha} : \alpha<\kappa\rangle\in \mathcal{N}_{1}$.
The sets $X_{\alpha}$ have fully symmetric names $\dot{X_{\alpha}}$ (any permutation
of a name for an element of $\dot{X_{\alpha}}$ returns a name for an element of $\dot{X_{\alpha}}$). Let $\dot{X}=\{\dot{X_{\alpha}} : \alpha<\kappa\}$. Consequently, $sym (\dot{X})\in \mathcal{F}$, i.e., $\dot{X}$ is symmetric with respect to $\mathcal{F}$. Since all the names appearing in $\dot{X}$ are from $HS$, $\dot{X}\in HS$. Consequently, $\langle X_{\alpha} : \alpha<\kappa\rangle\in \mathcal{N}_{1}$.

Although each $X_{\alpha}\not = \emptyset$, we claim that $(\prod_{\alpha<\kappa}X_{\alpha})^{\mathcal{N}_{1}}=\emptyset$. Otherwise let $y\in(\prod_{\alpha<\kappa}X_{\alpha})^{\mathcal{N}_{1}}$. Since $y$ is a sequence of sets of ordinals, so can be coded as a set of ordinals.
Then there is a $\gamma<\kappa$ such that $y\in V[G\restriction\gamma]$ by \textbf{Lemma 4.1} and $V[G\restriction\gamma]$ is $\mathbb{P}$-generic over $V$ such that $\vert\mathbb{P}\vert<\kappa$. So there is a final segment of the sequence $\langle (\kappa^{+++}_{\alpha}): \alpha<\kappa\rangle$ which remains a sequence of cardinals in $V[G\restriction\gamma]$ which is a contradiction. 
\end{proof}
\end{proof}

\begin{remark}
Since $\mathsf{GCH}$ implies $\mathsf{AC}$, $\mathsf{GCH}$ is weakened to a form which states that {\em there is no injection from $\delta^{++}$ into $\mathcal{P}(\delta)$} in \cite[\textbf{Theorem 3}]{Apt2012}. We follow this weakened version of $\mathsf{GCH}$ in our case. We follow the explanation given in \cite[\textbf{section 4.1}]{Kar2019} by Karagila, to see that in $\mathcal{N}_{1}$, $\mathsf{GCH}$ holds for a limit cardinal $\delta$ if and only if $\delta>\kappa$.

The referee suggested us to remark the following. In the context of $\mathsf{ZF}$, there are two reasonable definitions for the statement “$\mathsf{GCH}$ at $\mu$”.
\begin{enumerate}
\item There is no injection $\mu^{++} \rightarrow^{inj} \mathcal{P}(\mu)$.
\item There is no surjection $\mathcal{P}(\mu) \rightarrow^{sur} \mu^{++}$. 
\end{enumerate}
In $\mathsf{ZF}$, it is possible that there is no $\mu^{+}\rightarrow^{inj} \mathcal{P}(\mu)$, but there is always a surjection
$\mathcal{P}(\mu) \rightarrow^{sur} \mu^{+}$. In our case the above two definitions behave the
same, so the referee suggested us to remark that both definitions (1) and (2) work, by the same proof.
\end{remark}

\section{Proving Dimitriou's Conjecture}
Fix an arbitrary $n_{0}\in \omega$. Apter \cite[\textbf{Theorem 1}]{Apt1983a} obtained a model of $\mathsf{ZF+\neg DC_{\aleph_{n_{0}+1}}}$ where $\aleph_{\omega}$ carries a Rowbottom filter and $\mathsf{DC_{\aleph_{n_{0}}}}$ holds, from an $\omega$-sequence of measurable cardinals.
In \textbf{section 8}, we observe that there is an alternating sequence of measurable and non-measurable cardinals in that model. Apter constructed the model based on Easton support products of L\'{e}vy collapse. 
Consequently, $\mathsf{DC_{\aleph_{n_{0}}}}$ was preserved (cf. \cite[\textbf{Lemma 1.4}]{Apt1983a}). Dimitriou \cite[\textbf{section 1.4}]{Dim2011} constructed a similar model with an alternating sequence of measurable and non-measurable cardinals, excluding the singular limits. She constructed the model based on finite support products of collapsing functions, unlike the model from \cite{Apt1983a}. 
In \cite{Dim2011}, Dimitriou claimed that by using such a  finite support product construction, a lot of arguments could be made easier. In particular, she used finite support products of {\em injective tree-Prikry forcings}, in several constructions from \cite[\textbf{Chapter 2}]{Dim2011}. There are many models of $\mathsf{ZF}$ constructed using the finite support products of L\'{e}vy collapse.  Hayut and Karagila \cite[\textbf{Theorem 5.6}]{HK2019} considered a symmetric extension constructed using the finite support products of L\'{e}vy collapse.
In \textbf{section 6}, we encounter two models of $\mathsf{ZF}$ constructed using the finite support products of L\'{e}vy collapse due to Apter and Cody from \cite{AC2013} (cf. \cite[\textbf{Theorem 2}]{AH1991} as well).
On the other hand, there is a downside to this method. 
Specifically, Dimitriou conjectured that $\mathsf{DC_{\omega}}$ would fail in the model.
In this section, we prove that $\mathsf{AC_{\omega}}$ fails in the model and thus prove the conjecture of Dimitriou.  In other words, we prove \textbf{Theorem 1.6}. We refer the
reader to the terminologies from \textbf{section 2.4}.

\begin{proof}{\textbf{(of Theorem 1.6)}}
Firstly, we give a description of the symmetric extension constructed in \cite[\textbf{section  1.4}]{Dim2011} as follows.
\begin{enumerate}
\item \textbf{The ground model ($V$).}
Let $V$ be a model of $\mathsf{ZFC}$, $\rho$ be an ordinal, and $\mathcal{K}=\langle\kappa_{\epsilon}:0<\epsilon<\rho\rangle$ be a sequence of measurable cardinals. Let $\kappa_0$ be a regular cardinal below all the measurable cardinals in $\mathcal{K}$. 

\item \textbf{Defining the triple $(\mathbb{P},\mathcal{G},\mathcal{I})$.}  
\begin{itemize}
\item Let $\kappa_{1}' = \kappa_0$. For each $1<\epsilon< \rho$ we define the following cardinals,

$\kappa_{\epsilon}'= \kappa_{\epsilon-1}^{+}$ if $\epsilon$ is a successor ordinal,

$\kappa_{\epsilon}'= (\bigcup_{\zeta<\epsilon} \kappa_{\zeta})^+ $ if $\epsilon$ is a limit ordinal and $\bigcup_{\zeta<\epsilon} \kappa_{\zeta}$ is singular,

$\kappa_{\epsilon}'= (\bigcup_{\zeta<\epsilon} \kappa_{\zeta})^{++}$ if $\epsilon$ is a limit ordinal and $\bigcup_{\zeta<\epsilon} \kappa_{\zeta}=\kappa_{\epsilon}$ is regular,

$\kappa_{\epsilon}'= \bigcup_{\zeta<\epsilon} \kappa_{\zeta}$ if $\epsilon$ is a limit ordinal and $\bigcup_{\zeta<\epsilon} \kappa_{\zeta}<\kappa_{\epsilon}$ is regular.

Let $\mathbb{P} = \prod_{0<i<\rho} \mathbb{P}_{i}$ be the Easton support product of $\mathbb{P}_{i}=Fn(\kappa_{i}', \kappa_{i},\kappa_{i}')$ ordered componentwise where for each $0<i<\rho$, $Fn(\kappa_{i}', \kappa_{i},\kappa_{i}')$= $\{p:\kappa_{i}'\rightharpoonup\kappa_{i}:\vert p\vert<\kappa_{i}'$ and $p$ is an injection$\}$ ordered by reverse inclusion. We denote by $p:\kappa_{i}'\rightharpoonup\kappa_{i}$ a partial function from $\kappa_{i}'$ to $\kappa_{i}$.

\item Let $\mathcal{G}=
\prod_{0<i<\rho} \mathcal{G}_{i}$ where for each $0<i<\rho$, $\mathcal{G}_{i}$ is the full permutation group of $\kappa_{i}$ that can be extended to $\mathbb{P}_{i}$ by permuting the range of its conditions, i.e., for all $a\in \mathcal{G}_{i}$ and $p\in \mathbb{P}_{i}$, $a(p)=\{(\psi,a(\beta)):(\psi,\beta)\in p\}$.

\item For every finite sequence of ordinals $e=\{\alpha_i: 1\leq i\leq m\}$ such that for every  $1\leq i\leq m$ there is a distinct $0<\epsilon_{i}< \rho$ such that $\alpha_{i}\in (\kappa_{\epsilon_i}',\kappa_{\epsilon_i})$, we define $E_{e}=\{\langle  \emptyset,... ,p_{\epsilon_1}\cap (\kappa_{\epsilon_1}'\times \alpha_1),\emptyset,...,p_{\epsilon_2}\cap (\kappa_{\epsilon_2}'\times\alpha_2),\emptyset,...,p_{\epsilon_i}\cap (\kappa_{\epsilon_i}'\times\alpha_i),\emptyset,...,p_{\epsilon_m}\cap (\kappa_{\epsilon_m}'\times \alpha_m), \emptyset,...\rangle; \overrightarrow{p}\in \mathbb{P} \}$ and $\mathcal{I}=\{E_e : e\in \prod^{fin}_{0<i<\rho}(\kappa_{i}', \kappa_{i})\}$ where 
$\prod^{fin}_{0<i<\rho}(\kappa_{i}', \kappa_{i})$ is the finite support product.
\end{itemize}

\item \textbf{Defining the symmetric extension of $V$.} 
Clearly, $\mathcal{I}$ is a projectable symmetry generator with projections $\overrightarrow{p}\restriction E_{e}=\langle  \emptyset,..., p_{\epsilon_1}\cap (\kappa_{\epsilon_1}'\times \alpha_1),\emptyset,...,p_{\epsilon_2}\cap (\kappa_{\epsilon_2}'\times\alpha_2),\emptyset,...,p_{\epsilon_m}\cap (\kappa_{\epsilon_m}'\times\alpha_m),\emptyset,...\rangle$. Let $\mathcal{I}$ generate a normal filter $\mathcal{F}_{\mathcal{I}}$ over $\mathcal G$. Let $G$ be a $\mathbb{P}$-generic filter over $V$. We consider the symmetric model $V(G)^{\mathcal{F}_{\mathcal{I}}}$ as our desired symmetric extension. 
\end{enumerate}

It is possible to see that $\mathbb{P}$ is $(\mathcal{G},\mathcal{I})$-homogeneous and so $\langle\mathbb{P},\mathcal{G},\mathcal{I}\rangle$ has the approximation property.
Consequently, by \textbf{Lemma 2.16} for all set of ordinals $X\in V(G)^{\mathcal{F}_{\mathcal{I}}}$, there exists an $E_{e}\in \mathcal{I}$ such that $X\in V[G\cap E_{e}]$.
Following \cite[\textbf{Lemma 1.35}]{Dim2011}, for every $0<\epsilon<\rho$, $(\kappa'_{\epsilon})^{+}=\kappa_{\epsilon}$ in $V(G)^{\mathcal{F}_{\mathcal{I}}}$.
We prove that $\mathsf{AC_{\omega}}$ fails in $V(G)^{\mathcal{F}_{\mathcal{I}}}$. For the sake of convenience we call $V(G)^{\mathcal{F}_{\mathcal{I}}}$ as $V(G)$, $HS^{\mathcal{F}_{\mathcal{I}}}$ as $HS$, and $\mathcal{F}_{\mathcal{I}}$ as $\mathcal{F}$.

\begin{lem}
{\em In $V(G)$, $\mathsf{AC_{\omega}}$ fails.}
\end{lem}
\begin{proof}
Since the cardinality of $\kappa_{n}'$ is preserved in $V(G)$ for $n<\omega$, we can define in $V(G)$ the set $X_{n}=\{ x\subseteq \kappa_{n}' : x$ codes a well ordering of $((\kappa_{n}')^{+})^{V}$ of order type $\kappa_{n}'\}$. We claim that $\langle X_{n} : n<\omega\rangle\in V(G)$. 
The sets $X_{n}$ have fully symmetric names $\dot{X_{n}}$ (any permutation
of a name for an element of $\dot{X_{n}}$ returns a name for an element of $\dot{X_{n}}$). Let $\dot{X}=\{\dot{X_{n}} : n<\omega\}$. Consequently, $sym (\dot{X})\in \mathcal{F}$, i.e., $\dot{X}$ is symmetric with respect to $\mathcal{F}$. Since all the names appearing in $\dot{X}$ are from $HS$, $\dot{X}\in HS$. Consequently, $\langle X_{n} : n<\omega\rangle\in V(G)$.

Although $X_n\not = \emptyset$, we claim that $(\prod_{n<\omega}X_{n})^{V(G)}=\emptyset$. Otherwise let $y\in(\prod_{n<\omega}X_{n})^{V(G)}$. Since $y$ is a sequence of sets of ordinals, so can be coded as a set of ordinals. Thus, there is an $e=\{\alpha_1,...,\alpha_m\}$ such that $y\in V[G\cap E_{e}]$ by \textbf{Lemma 2.16}. There are distinct $\epsilon_{i}$ such that $\alpha_{i}\in (\kappa_{\epsilon_i}',\kappa_{\epsilon_i})$ and let $l$ be max$\{\epsilon_{i}: \alpha_{i}\in e\}$ such that $l$ is an integer. Next let $M=\{i:\epsilon_{i}\leq l\}$ and $M'=\{i:\epsilon_{i}> l\}$.
Then $V[G\cap E_{e}]$ is $\prod_{i\in M}Fn(\kappa'_{\epsilon_{i}}, \alpha_{i},\kappa'_{\epsilon_{i}})\times \prod_{i\in M'}Fn(\kappa'_{\epsilon_{i}}, \alpha_{i},\kappa'_{\epsilon_{i}})$-generic over $V$. By closure properties of $\prod_{i\in M'}Fn(\kappa'_{\epsilon_{i}}, \alpha_{i},\kappa'_{\epsilon_{i}})$, all elements of the sequence $\langle (\kappa'_{n})^{+}: n<\omega\rangle$ remain cardinals after forcing with $\prod_{i\in M'}Fn(\kappa'_{\epsilon_{i}}, \alpha_{i},\kappa'_{\epsilon_{i}})$.
Next, since $M$ is finite we can find $j<\omega$ such that for all $r\geq j$, $\vert\prod_{i\in M}Fn(\kappa'_{\epsilon_{i}}, \alpha_{i},\kappa'_{\epsilon_{i}})\vert<\kappa_{r}$. Thus, a final segment of the sequence $\langle (\kappa'_{n})^{+}: n<\omega\rangle$ remains a sequence of cardinals in $V[G\cap E_{e}]$ which is a contradiction. 
\end{proof}
\end{proof}

\begin{remark}
Hayut and Karagila \cite[\textbf{Theorem 5.6}]{HK2019} proved the following.
\begin{itemize}
    \item Assuming the existence of countably many measurable cardinals, it is consistent that there is a uniform ultrafilter on $\aleph_{\omega}$ but for all $0 < n < \omega$, there
are no uniform ultrafilters on $\aleph_{n}$.
\end{itemize}

They considered a symmetric extension $M$ based on finite support product of the symmetric collapses $Col(\kappa_{n}, < \kappa_{n+1})$. Following the proof of \textbf{Lemma 5.1}, we can say that $\mathsf{AC_{\omega}}$ fails in the symmetric extension $M$. We consider another similar symmetric extension. Let $V_{1}$ be a model of $\mathsf{ZFC}$ where $\langle \kappa_{n} : 1\leq n<\omega\rangle$ is a countable sequence of supercompact cardinals. Let $\mathbb{Q}$ be the forcing notion (see  \cite{Apt1983, Apt2005}) which makes the supercompactness of each $\kappa_{n}$ indestructible under $\kappa_{n}$-directed closed forcing notions. Let $H$ be a $\mathbb{Q}$-generic filter over $V_{1}$ and $V=V_{1}[H]$ be our ground model. Let $\kappa_{0}=\omega$ in $V$. 
Consider the symmetric extension $\mathcal{N}$ obtained by taking the finite support product of
the symmetric collapses $Col(\kappa_{n}, < \kappa_{n+1})$. In the resulting model $\mathcal{N}$ the following hold:
\begin{enumerate}
    \item Since the forcing notions involved are weakly homogeneous, if $A$ is a set of ordinals in $\mathcal{N}$, then $A$ was added by an intermediate submodel where $\mathsf{AC}$ holds. 
    \item For $n > 0$, each $\kappa_{n}$ becomes $\aleph_{n}$ in $\mathcal{N}$.
\end{enumerate}
Following \cite[\textbf{Theorem 4.3}]{HK2019}, we can observe that for each $1\leq n<\omega$, there are no uniform ultrafilters on $\aleph_{n}$ in $\mathcal{N}$. Consequently for each $1\leq n<\omega$, $\aleph_{n}$ can not be a measurable cardinal in $\mathcal{N}$. Since we are considering symmetric extension based on finite support products, $\mathsf{AC_{\omega}}$ fails following the proof of \textbf{Lemma 5.1}.
We can see that each $\aleph_{n}$ remains a Ramsey cardinal for $1 \leq n<\omega$ in $\mathcal{N}$. Fix $1\leq n<\omega$. Let $f:[\kappa_{n}]^{<\omega}\rightarrow 2$ is in $\mathcal{N}$. Since $f$ can be coded by a set of ordinals, $f$ was added by an intermediate submodel (say $V'$) where $\mathsf{AC}$ holds. Without loss of generality, we can say that $V'= V[G_1][G_2]$ where $G_1$ is $\mathbb{Q}_1$-generic over $V$ such that $\mathbb{Q}_1$ is $\kappa_{n}$-directed closed and $G_2$ is $\mathbb{Q}_2$-generic over $V[G_{1}]$ such that $\vert\mathbb{Q}_2\vert<\kappa_{n}$. Since $\mathbb{Q}_{1}$ is $\kappa_{n}$-directed closed, $\kappa_{n}$ remains supercompact in $V[G_{1}]$ as the supercompactness of $\kappa_{n}$ was indestructible under $\kappa_{n}$-directed closed forcing notions in $V$.
Consequently, $\kappa_{n}$ remains a Ramsey cardinal in $V[G_{1}]$. Since $\vert\mathbb{Q}_2\vert<\kappa_{n}$, $\kappa_{n}$ remains Ramsey in $V[G_{1}][G_{2}]$ by \textbf{Theorem 2.5}. There is then a set $X\in [\kappa_{n}]^{\kappa_{n}}$ homogeneous for $f$ in $V'$, and since $V'\subseteq \mathcal{N}$, $X\in[\kappa_{n}]^{\kappa_{n}}$ is homogeneous for $f$ in $\mathcal{N}$. Consequently, for each $1\leq n<\omega$, $\kappa_{n}$ is Ramsey in $\mathcal{N}$.
\end{remark}

\section{Reducing the assumption of supercompactness to strong compactness}

\subsection{Strongly compact Prikry forcing}
Suppose $\lambda>\kappa$ and $\kappa$ be a $\lambda$-strongly compact cardinal in the ground model $V$. Let $\mathcal{U}$ be a $\kappa$-complete fine ultrafilter over $\mathcal{P}_{\kappa}(\lambda)$.

\begin{defn}{(cf. \cite[\textbf{Definition 1.51}]{Git2010})}
A set $T$ is called a {\em $\mathcal{U}$-tree with trunk t} if and only if the following hold.
\begin{enumerate}
    \item $T$ consists of finite sequences $\langle P_{1},...,P_{n}\rangle$ of elements of $\mathcal{P}_{\kappa}(\lambda)$ so that $P_{1}\subseteq P_{2}\subseteq ... \subseteq P_{n}$.
    \item $\langle T, \unlhd\rangle$ is a tree, where $\unlhd$ is the order of the end extension of finite sequences.
    \item $t$ is a trunk of $T$, i.e., $t\in T$ and for every $\eta\in T$, $\eta \unlhd t$ or $t \unlhd \eta$.
    \item For every $t\unlhd \eta$, $Suc_{T} (\eta)=\{Q\in \mathcal{P}_{\kappa}(\lambda): \eta \frown\langle Q\rangle\in T\}\in \mathcal{U}$.
\end{enumerate}
\end{defn}
The set $\mathbb{P}_{\mathcal{U}}$ consists of all pairs $\langle t,T\rangle$ such that $T$ is a
$\mathcal{U}$-tree with trunk $t$. If $\langle t,T\rangle, \langle s,S\rangle\in\mathbb{P}_{\mathcal{U}}$, we say that $\langle t,T\rangle$ is stronger than
$\langle s,S\rangle$, and denote this by $\langle t,T\rangle\geq\langle s,S\rangle$, if and only if $T \subseteq S$. We call $\mathbb{P}_{\mathcal{U}}$ with the ordering defined above as the {\em strongly compact Prikry forcing} with respect to $\mathcal{U}$.\footnote{Alternatively, we also recall the definition of a strongly compact Prikry forcing $\mathbb{P}_{\mathcal{U}}$ from \cite{AH1991}. Let $\mathcal{U}$ be a fine measure on $\mathcal{P}_{\kappa}(\lambda)$ and $\mathcal{F}=\{f: f$ is a function from $[\mathcal{P}_{\kappa}(\lambda)]^{<\omega}$ to $\mathcal{U}\}$.  In particular, $\mathbb{P}_{\mathcal{U}}$ is the set of all finite sequences of the form $\langle p_{1},...p_{n},f\rangle$ satisfying the following properties.

\begin{itemize}
    \item $\langle p_{1},...p_{n}\rangle\in [\mathcal{P}_{\kappa}(\lambda)]^{<\omega}$.
    \item for $0\leq i<j\leq n$, $p_{i}\cap \kappa\not= p_{j}\cap \kappa$.
    \item $f\in \mathcal{F}$.
\end{itemize}

The ordering on $\mathbb{P}_{\mathcal{U}}$ is given by $\langle q_{1},...q_{m},g\rangle \leq \langle p_{1},...,p_{n},f\rangle$ if and only if we have the following.

\begin{itemize}
    \item $n\leq m$.
    \item $\langle p_{1},...,p_{n}\rangle$ is the initial segment of $\langle q_{1},...,q_{m}\rangle$.
    \item For $i=n+1,..., m$, $q_{i}\in f(\langle p_{1},...,p_{n}, q_{n+1},...,q_{i-1}\rangle)$.
    \item For $\overrightarrow{s}\in [\mathcal{P}_{\kappa}(\lambda)]^{<\omega}$, $g(\overrightarrow{s})\subseteq f(\overrightarrow{s})$.
\end{itemize}

For any regular $\delta\in [\kappa,\lambda]$, we denote $r\restriction{\delta}=\{\langle p_{0}\cap \delta,...p_{n}\cap\delta\rangle: \exists f\in \mathcal{F} \left[\langle p_{0},...p_{n},f\rangle\in G\right]\}$. In $V[r\restriction \kappa]\subseteq V[G]$, $\kappa$ is a singular cardinal having cofinality $\omega$. Since any two conditions having the same stems are compatible, i.e. any two conditions of the form $\langle p_{1},...,p_{n},f\rangle$ and $\langle p_{1},...,p_{n},g \rangle$ are compatible., \textbf{$\mathbb{P}_{\mathcal{U}}$ is $(\lambda^{<\kappa})^{+}$-c.c.}
}

Let $G$ be a $\mathbb{P}_{\mathcal{U}}$-generic filter over $V$. We summarize the necessary properties of $\mathbb{P}_{\mathcal{U}}$ from \cite{Git2010}.
\begin{itemize}
    \item By a Prikry like lemma and a similar proof as in the ordinary Prikry forcing,\footnote{i.e., the arguments of \cite[\textbf{Lemma 1.9}]{Git2010}.} $\mathbb{P}_{\mathcal{U}}$ does not add new bounded subsets to $\kappa$ (cf. \cite[\textbf{Lemma 1.1}]{AH1991}, \cite[\textbf{Theorem 1.52}]{Git2010}).
    
    \item Every cardinal in $(\kappa, \lambda]$ is collapsed to have size $\kappa$ in $V[G]$ (cf. \cite[\textbf{Lemma 1.50}]{Git2010} and the arguments before \cite[\textbf{Theorem 1.52}]{Git2010}).
    
    \item Every $\delta \in [\kappa, \mu]$ of cofinality $\geq \kappa$ (in $V$) changes its cofinality to $\omega$ in $V[G]$ (cf. \cite[\textbf{Lemma 1.50, Theorem 1.52}]{Git2010}), where $\mu=\lambda$ if $cf(\lambda) \geq \kappa$ and $\mu=\lambda^{+}$ if $cf(\lambda) < \kappa$.
    
    %\begin{center}$\mu=\begin{cases}\lambda,  & \text{if }cf(\lambda) \geq \kappa\\\lambda^{+},  & \text{if } cf(\lambda) < \kappa.\end{cases}$\end{center}
    
    \item Any two conditions with the same trunk, i.e. of the form $\langle t, T\rangle$ and $\langle t, S\rangle$ are compatible. Also there are $\lambda^{<\kappa}$ many possibilities for trunks for members of $\mathbb{P}_{\mathcal{U}}$. Consequently, $\mathbb{P}_{\mathcal{U}}$ satisfies the $(\lambda^{<\kappa})^{+}$-c.c. (cf. \cite[\textbf{Lemma 1.48}]{Git2010} and the arguments before \cite[\textbf{Theorem 1.52}]{Git2010}).
\end{itemize}

\subsection{The Proofs of Observations 1.7 and 1.8} In this subsection, we prove \textbf{Observation 1.7} and \textbf{Observation 1.8}. 
\begin{proof}{\textbf{(of Observation 1.7)}} We perform the construction in two stages. In the first stage, we consider a model similar to the choiceless model constructed in \cite[\textbf{section 2}]{AH1991}.

\begin{enumerate}
\item \textbf{The  ground  model ($V$):} We start with a model $V_{0}$ of $\mathsf{ZFC}$ where $\kappa$ is a strongly compact cardinal, $\theta$ is an ordinal, and $\mathsf{GCH}$ holds. By \cite[\textbf{Theorem 3.1}]{ADU2021} we can obtain a forcing extension $V_{1}$ in which the strong compactness of $\kappa$ is indestructible under $Add(\kappa, \theta)$ for all $\theta$. Then forcing with $Add(\kappa, \theta)$, we may assume without loss of generality that $\kappa$ is
strongly compact and $2^{\kappa}=\theta$ in a forcing extension $V$ of $V_{1}$. Let $\lambda$ be a cardinal in $V$ such that $\kappa < \lambda$ and $(cf(\lambda))^{V}<\kappa$.

\item \textbf{Defining an inner model of a forcing extension of $V$:}  

\begin{itemize}
\item Let $\mathcal{U}$ be a fine measure on $\mathcal{P}_{\kappa}(\lambda)$ and $\mathbb{P}=\mathbb{P}_{\mathcal{U}}$ be the strongly compact Prikry forcing. Let $G$ be a $\mathbb{P}_{\mathcal{U}}$-generic filter over $V$.

\item We consider a model similar to the choiceless model constructed in \cite[\textbf{section 2}]{AH1991}. In particular, we consider our model $\mathcal{N}$ to be the least model of $\mathsf{ZF}$ extending $V$ and containing
$r\restriction \delta$ for each regular $\delta\in [\kappa,\lambda)$
where $r\restriction{\delta}=\{\langle p_{0}\cap \delta,...,p_{n}\cap\delta\rangle: \exists f\in \mathcal{F} \left[\langle p_{0},...,p_{n},f\rangle\in G\right]\}$ but not the $\lambda$-sequence of $r\restriction\delta$'s. 
\end{itemize}
\end{enumerate}

We follow the homogeneity of strongly compact Prikry forcing mentioned in \cite[\textbf{Lemma 2.1}]{AH1991} to observe the following lemma.

\begin{lem}
{\em If $A\in \mathcal{N}$ is a set of ordinals, then $A\in V[r\restriction \delta]$ for some regular $\delta\in[\kappa,\lambda)$.}
\end{lem}

\begin{lem}
{\em In $\mathcal{N}$, $\kappa$ is a strong limit cardinal.}
\end{lem}

\begin{proof}
Since $V\subseteq \mathcal{N}\subseteq V[G]$ and $\mathbb{P}$ does not add bounded subsets to $\kappa$, $V$ and $\mathcal{N}$ have the same bounded subsets of $\kappa$.\footnote{We can observe another argument from \cite[\textbf{Lemma 2.2}]{AH1991}.} Consequently, in $\mathcal{N}$, $\kappa$ is a limit of inaccessible cardinals and thus a strong limit cardinal as well.
\end{proof}

As explained in the introduction, our definitions of “strong limit cardinal” and “inaccessible cardinal” generally do not make sense in choiceless models. In spite of that, we can see that the assertion in \textbf{Lemma 6.3} makes sense (see the paragraph after \cite[\textbf{Theorem 1}]{AC2013}). Since $\mathcal{N}$ and $V$ have the same bounded subsets of $\kappa$, the usual definitions of “$\kappa$ is a strong limit cardinal”
and “$\delta < \kappa$ is an inaccessible cardinal” make sense in $\mathcal{N}$.

\begin{lem}
{\em If $\gamma\geq\lambda$ is a cardinal in $V$, then $\gamma$ remains a cardinal in $\mathcal{N}$.}
\end{lem}

\begin{proof}
For the sake of contradiction, let $\gamma$ is not a cardinal in $\mathcal{N}$. There is then a bijection $f:\alpha\rightarrow\gamma$ for some $\alpha<\gamma$ in $\mathcal{N}$. Since $f$ can be coded by a set of ordinals, by \textbf{Lemma 6.2} we have $f\in V[r\restriction\delta]$ for some regular $\delta\in [\kappa,\lambda)$. Since $\mathsf{GCH}$ is assumed in $V_{0}$ we have $(\delta^{<\kappa})^{V_{0}}=\delta$, and since $Add(\kappa,\theta)$ preserves cardinals and adds no sequences of ordinals of length less than $\kappa$, we conclude that $(\delta^{<\kappa})^{V}=(\delta^{<\kappa})^{V_{0}}=\delta$.     
Now $\mathbb{P}_{\mathcal{U}\restriction \delta}$ is $(\delta^{<\kappa})^{+}$-c.c. in $V$ and hence $\delta^{+}$-c.c. in $V$. Consequently, $\gamma$ is a cardinal in $V[r\restriction\delta]$ which is a contradiction.
\end{proof}

\begin{lem}
{\em In $\mathcal{N}$, $cf(\kappa)=\omega$. Moreover, $(\kappa^{+})^{\mathcal{N}}=\lambda$ and $cf(\lambda)^{\mathcal{N}} = cf(\lambda)^V$.}
\end{lem}

\begin{proof}
For each regular $\delta\in [\kappa,\lambda)$, we have $V[r\restriction \delta]\subseteq \mathcal{N}$. Consequently, $cf(\kappa)^{\mathcal{N}}=\omega$ since $cf(\kappa)^{V[r\restriction \kappa]}=\omega$.
Following \cite[\textbf{Lemma 2.4}]{AH1991}, every ordinal in $(\kappa,\lambda)$ which is a cardinal in $V$ collapses to have size $\kappa$ in $\mathcal{N}$, and so $(\kappa^{+})^{\mathcal{N}}=\lambda$. 
Since $V$ and $\mathcal{N}$ have same bounded subsets of $\kappa$, we see that $cf(\lambda)^{\mathcal{N}}=cf(\lambda)^{V}<\kappa$. 
\end{proof}

We can see that since, $V\subseteq \mathcal{N}$ and $(2^{\kappa}=\theta)^{V}$, there is a $\theta$-sequence of distinct subsets of $\kappa$ in $\mathcal{N}$. Since $cf(\kappa^{+})^{\mathcal{N}}<\kappa$ we can also see that $\mathsf{AC_{\kappa}}$ fails in $\mathcal{N}$.

In the second stage, we consider an inner model of a forcing extension of $\mathcal{N}$ based on a product of L\'{e}vy collapses as done in the proof of \cite[\textbf{Theorem 2}]{AC2013}. 
\begin{enumerate}
\item \textbf{The ground model ($\mathcal{N}$).} 
Consider the ground model to be $\mathcal{N}$. Let $\langle \kappa_{n}:n<\omega\rangle$ be a sequence of inaccessible cardinals less than $\kappa$ which is cofinal in $\kappa$.

\item \textbf{Defining an inner model of a forcing extension of $\mathcal{N}$.}  

\begin{itemize}
\item Let $\mathbb{P}=Col(\omega,<\kappa)$, and let $G$ be $\mathbb{P}$-generic over $\mathcal{N}$. Let $\mathbb{P}_{n}=Col(\omega,<\kappa_{n})$. Following the proof of \cite[\textbf{Theorem 2}]{AC2013}, $G_{n}=G\cap \mathbb{P}_{n}$ is $\mathbb{P}_{n}$-generic over $\mathcal{N}$.
\item Let $\mathcal{M}$ be the least model of $\mathsf{ZF}$ extending $\mathcal{N}$ containing each $G_{n}$, but not $G$ as constructed in \cite[\textbf{Theorem 2}]{AC2013}. 

\end{itemize}
\end{enumerate}

Following the proof of \cite[\textbf{Theorem 2}]{AC2013}, we have the following in $\mathcal{M}$.

\begin{enumerate}
\item Since $\mathcal{M}$ contains $G_{n}$ for each $n$, cardinals in $[\omega,\kappa)$ are collapsed to have size $\omega$ and so $\aleph_{1}^{\mathcal{M}}\geq\kappa$. 
\item If $x \in \mathcal{M}$ is a set of
ordinals, then $x \in \mathcal{N}[G_{n}]$ for some $n < \omega$.
\item  Since $Col(\omega, <\kappa_{n})$ is canonically well-orderable in $\mathcal{N}$ with order type $\kappa_{n}$, cardinals and cofinalities greater than or equal to $\kappa$ are preserved to $\mathcal{N}[G_{n}]$.
\item Since $\kappa$ is not collapsed, $\kappa=\aleph_{1}^{\mathcal{M}}$, $cf(\aleph_{1})^{\mathcal{M}}=cf(\aleph_{2})^{\mathcal{M}}=\omega$. Consequently, $\mathsf{AC_{\omega}}$ fails in $\mathcal{M}$.
    \item There is a sequence of distinct subsets of $\aleph_{1}$ of length $\theta$.
\end{enumerate}
\end{proof}
\begin{proof}{\textbf{(of Observation 1.8)}} We recall the model $\mathcal{N}$ from the previous proof.  We consider an inner model of the forcing extension of $\mathcal{N}$ as done in the proof of \cite[\textbf{Theorem 3}]{AC2013}.
\begin{enumerate}
\item \textbf{The ground model ($\mathcal{N}$).} 
Consider the ground model to be $\mathcal{N}$ as in the previous proof. Let $\langle \kappa_{n}:n<\omega\rangle$ be a sequence of inaccessible cardinals less than $\kappa$ which is cofinal in $\kappa$. 
\item \textbf{Defining an inner model of a forcing extension of $\mathcal{N}$.}  

\begin{itemize}
\item Let $\mathbb{P}_{0}=Col(\omega,<\kappa_{0})$, $\mathbb{P}_{i}=Col(\kappa_{i-1},<\kappa_{i})$ for $1\leq i<\omega$. Let $\mathbb{P}=\Pi_{i<\omega}^{fin} \mathbb{P}_{i}$ be a finite support product of $\mathbb{P}_{i}$. For each $n<\omega$, we can factor $\mathbb{P}$ as $\mathbb{P}\cong \mathbb{P}_{n}^{*}\times \mathbb{P}^{n}$ where $\mathbb{P}_{n}^{*}=\Pi_{0\leq i\leq n}^{fin} \mathbb{P}_{i}$ and $\mathbb{P}^{n}=\Pi_{n+1\leq i<\omega}^{fin} \mathbb{P}_{i}$. Let $G\cong G_{n}^{*}\times G^{n}$ be a $\mathbb{P}$-generic filter over $\mathcal{N}$. 
Following \cite[\textbf{Theorem 3}]{AC2013}, each $G_{n}^{*}$ is $\mathbb{P}_{n}^{*}$-generic over $\mathcal{N}$.

\item Let $\mathcal{M}$ be the least model of $\mathsf{ZF}$ extending $\mathcal{N}$ containing each $G_{n}^{*}$, but not $\langle G^{*}_{n}:n<\omega \rangle$ as constructed in \cite[\textbf{Theorem 3}]{AC2013}.

\end{itemize}
\end{enumerate}

Following the proof of \cite[\textbf{Theorem 3}]{AC2013}, we have the following in $\mathcal{M}$.

\begin{enumerate}
\item Since $G^{*}_{n} \in \mathcal{M}$ for each $n < \omega$, we have $\aleph_{\omega} \geq \kappa$ and hence $\aleph_{\omega+1} \geq (\kappa^{+})^{\mathcal{N}}$ in $\mathcal{M}$. 
\item If $x$ is a set of ordinals in $\mathcal{M}$, then $x \in  \mathcal{N}[G^{*}_{n}]$ for some
$n < \omega$ (cf. \cite[\textbf{Lemma 6}]{AC2013}).
\item Since $\mathcal{N}$
and $V$ contain the same bounded subsets of $\kappa$, and $V \subseteq \mathcal{N}$, $\mathbb{P}^{*}_{n}$
can be well-ordered in both $V$ and $\mathcal{N}$ with order type less than $\kappa$. Therefore,
cardinals and cofinalities greater
than or equal to $\kappa$ are preserved.
    \item $\kappa=\aleph_{\omega}$ and $(\kappa^{+})^{\mathcal{N}}=\aleph_{\omega+1}$ are both singular with $\omega\leq cf(\aleph_{\omega+1})<\aleph_{\omega}$. 
    \item $\mathsf{AC_{\omega}}$ fails in $\mathcal{M}$.
    \item There is a sequence of distinct subsets of $\aleph_{\omega}$ of length $\theta$.
\end{enumerate}
\end{proof}

\section{Infinitary Chang conjecture from a measurable cardinal}

\subsection{Infinitary Chang conjecture}
We recall the required definitions and relevant lemmas from \cite[\textbf{Chapter 3}]{Dim2011}. For the sake of our convenience we denote a structure $\mathcal{A}$ on domain $A$ as $\mathcal{A}=\langle A,...\rangle$.

\begin{defn}
{(\textbf{Set of good indiscernibles}; \cite[\textbf{Definition 3.2}]{Dim2011})} For a structure $\mathcal{A}=\langle A,...\rangle$ with $A\subseteq Ord$, a set $I\subseteq A$ is a {\em set of indiscernibles} if for all $n<\omega$, all $n$-ary formula $\phi$ in the language for $\mathcal{A}$, and every $\alpha_1,...,\alpha_n,\alpha'_1,...,\alpha'_n$ in $I$, if $\alpha_1<...<\alpha_n$ and $\alpha'_1<...<\alpha'_n$ then 

\centerline{$\mathcal{A}\models \phi(\alpha_1,...,\alpha_n)$ if and only if $\mathcal{A}\models \phi(\alpha'_1,...,\alpha'_n)$.}

The set $I$ is a {\em set of good indiscernibles} if and only if it is a set of indiscernibles and we allow parameters that lie below min$\{\alpha_1,...,\alpha_n,\alpha'_1,...,\alpha'_n \}$, i.e., if for every $x_1,...,x_m\in A$ such that $x_1,...,x_m\leq min\{\alpha_1,...,\alpha_n,\alpha'_1,...,\alpha'_n\}$ and every $(n+m)$-ary formula $\phi$,

\centerline{$\mathcal{A}\models \phi(x_1,...,x_m,\alpha_1,...,\alpha_n)$ if and only if $\mathcal{A}\models \phi(x_1,...,x_m,\alpha'_1,...,\alpha'_n)$.} 
\end{defn}

\begin{defn}
{(\textbf{$\alpha$-Erd\H{o}s cardinal}; \cite[\textbf{Definition 0.14}]{Dim2011})} The partition relation $\alpha\rightarrow (\beta)^{\gamma}_{\delta}$ for ordinals $\alpha,\beta,\gamma,\delta$ means for all $f:[\alpha]^{\gamma}\rightarrow\delta$ there is an $X\in[\alpha]^{\beta}$ such that $X$ is homogeneous for $f$. For an infinite ordinal $\alpha$, the {\em $\alpha$-Erd\H{o}s cardinal} $\kappa(\alpha)$ is the least $\kappa$ such that $\kappa\rightarrow(\alpha)^{<\omega}_{2}$.
\end{defn}

For cardinals $\kappa>\lambda$ and ordinal $\theta<\kappa$ we mean $\kappa\rightarrow^{\theta}(\lambda)^{<\omega}_{2}$ if for every first order structure $\mathcal{A}=\langle \kappa,...\rangle$ with a countable language, there is a set $I\in [\kappa\backslash\theta]^{\lambda}$ of good indiscernibles for $\mathcal{A}$ (cf. \cite[\textbf{Definition 3.7}]{Dim2011}).

\begin{defn}
{(\textbf{Infinitary Chang conjecture};  \cite[\textbf{Definition 3.10}]{Dim2011})} 
Let $\langle \kappa_{n}: n < \omega\rangle$ and $\langle \lambda_{n}: n < \omega\rangle$ be two increasing sequences of cardinals such that $\kappa_{n}>\lambda_{n}$ for every $n<\omega$. The {\em infinitary Chang conjecture} for these sequences, written as $(\kappa_n)_{n\in\omega}\twoheadrightarrow (\lambda_n)_{n\in\omega}$, is the statement ``for every  first order structure $\mathcal{A}=\langle \bigcup_{n<\omega}\kappa_{n},... \rangle$ there is an elementary substructure $\mathcal{B}\prec\mathcal{A}$ with domain $B$ of cardinality $\bigcup_{n<\omega}{\lambda_{n}}$ such that for every $n\in\omega$, $\vert B\cap \kappa_n\vert=\lambda_{n}$".
\end{defn}

\begin{defn}{(cf. \cite[\textbf{Definition 3.14}]{Dim2011})}
Let $\langle \kappa_{i}: i<\omega\rangle$ and $\langle \lambda_{i}: 0< i<\omega\rangle$ be two increasing sequences of cardinals. Let $\kappa=\bigcup_{i<\omega} \kappa_{i}$. We say $\langle\kappa_{i}: i<\omega\rangle$ is a {\em coherent sequence of cardinals} with the property $\kappa_{i+1}\rightarrow^{\kappa_i}(\lambda_{i+1})^{<\omega}_{2}$ if and only if for every structure $\mathcal{A}=\langle \kappa,...\rangle$ with a countable language, there is a $\langle \lambda_{i}: 0< i<\omega\rangle$-coherent sequence of good indiscernibles for $\mathcal{A}$ with respect to $\langle \kappa_{i}: i<\omega\rangle$, i.e., a sequence $\langle A_{i} : 0  < i < \omega\rangle$ with the following properties.
\begin{enumerate}
\item for every $0<i<\omega$, $A_{i} \in
[\kappa_{i}\backslash \kappa_{i-1}]^{\lambda_{i}}$,

\item if $x, y \in [\kappa]^{<\omega}$ are such that $x = \{x_{1},..., x_{n}\}$, $y = \{y_{1},..., y_{n}\}$, $x, y \subseteq \bigcup_{0<i<\omega} A_{i}$, and for every $0<i<\omega$, $\vert x \cap A_{i}\vert = \vert y \cap A_{i}\vert$ then for every $(n + l)$-ary formula $\phi$ in the language of $\mathcal{A}$ and every $z_{1},...,z_{l} < min \{x_{1},...,x_{n},y_{1},...,y_{n}\}$,
$\mathcal{A} \models \phi(z_{1},...,z_{l},x_{1},...,x_{n}) \iff \mathcal{A} \models \phi(z_{1},...,z_{l},y_{1},..., y_{n})$.
\end{enumerate}
\end{defn}

\begin{lem}
{($\mathsf{ZF}$; \textbf{Dimitriou}; \cite[\textbf{Corollary 3.15}]{Dim2011})}{\em Let $\langle \kappa_{i}: i<\omega\rangle$ and $\langle \lambda_{i}: 0< i<\omega\rangle$ be two increasing sequences of cardinals. Let $\kappa=\bigcup_{i<\omega} \kappa_{i}$. If $\langle\kappa_{i}: i<\omega\rangle$ is a coherent sequence of cardinals with the property $\kappa_{i+1}\rightarrow^{\kappa_i}(\lambda_{i+1})^{<\omega}_{2}$ then the infinitary Chang Conjecture $(\kappa_n)_{n\in\omega}\twoheadrightarrow (\lambda_n)_{n\in\omega}$ holds.}
\end{lem}

\begin{lem} 
{(\textbf{Dimitriou}; \cite[\textbf{Proposition 3.50}]{Dim2011})} {\em Let us assume that $V\models$ $\mathsf{ZFC}+``\kappa=\kappa(\lambda)$ exists", $\mathbb{P}$ is a forcing notion such that $\vert \mathbb{P}\vert<\kappa$, and $\mathbb{Q}$ is a forcing notion that doesn't add subsets to $\kappa$. If $G$ is $\mathbb{P}\times\mathbb{Q}$ generic then for every $\theta<\kappa$, $V[G]\models\kappa\rightarrow^{\theta}(\lambda)^{<\omega}_{2}$.}
\end{lem}

\begin{lem}
{(\textbf{Dimitriou}; \cite[\textbf{Lemma 3.52}]{Dim2011})} {\em Let $\langle \kappa_{i}: i<\omega\rangle$ and $\langle \lambda_{i}: 0< i<\omega\rangle$ be two increasing sequences of cardinals such that $\langle \kappa_{i}: 0< i<\omega\rangle$ is a coherent sequence of Erd\H{o}s cardinals with respect to $\langle \lambda_{i}: 0< i<\omega\rangle$. If $\mathbb{P}_1$ is a forcing notion of cardinality $<\kappa_1$ and G is $\mathbb{P}_1$-generic, then in $V[G]$, $\langle \kappa_{i}: i<\omega\rangle$ is a coherent sequence of cardinals with the property $\kappa_{i+1}\rightarrow^{\kappa_i}(\lambda_{i+1})^{<\omega}_{2}$.}
\end{lem}

\subsection{The Proofs of Theorems 1.9 and 1.10} In this subsection, we prove \textbf{Theorem 1.9} and \textbf{Theorem 1.10}. 
\begin{proof}{\textbf{(of Theorem 1.9)}}
\begin{enumerate}
\item \textbf{The ground model ($V$).}
Let $\kappa$ be a measurable cardinal in a model $V'$ of $\mathsf{ZFC}$. By Prikry forcing it is possible to make $\kappa$ singular with cofinality $\omega$ where an end segment $\langle \kappa_{i}: 1\leq i<\omega\rangle$ of the Prikry sequence $\langle \delta_{i}: 1\leq i <\omega\rangle$ is a {\em coherent} sequence of Ramsey cardinals by \cite[\textbf{Theorem 3}]{AK2006}. Now Ramsey cardinals $\kappa_{i}$ are exactly the $\kappa_{i}$-Erd\H{o}s cardinals. Thus we obtain a generic extension (say $V$) where $\langle\kappa_i: 1\leq i<\omega\rangle$ is a coherent sequence of cardinals with supremum $\kappa$ such that for all $1\leq i<\omega$, $\kappa_i=\kappa(\kappa_i)$. We define the following cardinals.

\begin{enumerate}
    \item $\kappa'_{0}=\omega$ and $\kappa_{0}=\aleph_{\omega}$.
    \item $\kappa'_{1}=\aleph_{\omega+1}$.
    \item $\kappa'_{i}=\kappa_{i-1}^{+\omega_{i-1}+1}$for each $1< i<\omega$.
\end{enumerate}

\item \textbf{Defining the triple $\langle\mathbb{P},\mathcal{G},\mathcal{I}\rangle$.} 
We consider a triple similar to the one constructed in \textbf{section 5}.

\begin{itemize}
\item Let
$\mathbb{P} =\prod_{i<\omega}\mathbb{P}_{i}$ be the Easton support product\footnote{In this case, it is equivalent to full
support.} of $\mathbb{P}_{i}=Fn(\kappa_{i}', \kappa_{i},\kappa_{i}')$ ordered componentwise where for each $i<\omega$, $Fn(\kappa_{i}', \kappa_{i},\kappa_{i}')$= $\{p:\kappa_{i}'\rightharpoonup\kappa_{i}:\vert p\vert<\kappa_{i}'$ and $p$ is an injection$\}$ ordered by reverse inclusion. 

\item
$\mathcal{G}=
\prod_{i<\omega} \mathcal{G}_{i}$ where for each $i<\omega$, $\mathcal{G}_{i}$ is the full permutation group of $\kappa_{i}$ that can be extended to $\mathbb{P}_{i}$ by permuting the range of its conditions, i.e., for all $a\in \mathcal{G}_{i}$ and $p\in \mathbb{P}_{i}$, $a(p)=\{(\psi,a(\beta)):(\psi,\beta)\in p\}$.

\item 
For $m\in\omega$ and $e=\{\alpha_1,...,\alpha_m\}$ a sequence of ordinals such that for each  $1\leq i\leq m$, there is a distinct $\epsilon_{i}<\omega$ such that $\alpha_{i}\in (\kappa'_{\epsilon_i},\kappa_{\epsilon_i})$, 
we define $E_{e}=\{\langle \emptyset,...,p_{\epsilon_1}\cap (\kappa'_{\epsilon_1}\times \alpha_1),\emptyset,...,p_{\epsilon_2}\cap (\kappa'_{\epsilon_2}\times\alpha_2),\emptyset,...,p_{\epsilon_m}\cap (\kappa'_{\epsilon_m}\times\alpha_m),\emptyset,...\rangle; \overrightarrow{p}\in \mathbb{P} \}$ and $\mathcal{I}=\{E_e : e\in \prod^{fin}_{i<\omega}(\kappa'_{i}, \kappa_{i})\}$.

\end{itemize}
\item \textbf{Defining symmetric extension of $V$.} 
Let $\mathcal{I}$ generate a normal filter $\mathcal{F}_{\mathcal{I}}$ over $\mathcal G$.
Let $G$ be a $\mathbb{P}$-generic filter. We consider the symmetric model $V(G)^{\mathcal{F}_I}$. We denote $V(G)^{\mathcal{F}_I}$ by $V(G)$ for the sake of convenience.
\end{enumerate}

Since the forcing notions involved are weakly homogeneous, the following holds. 

\begin{lem}
{\em If $A\in V(G)$ is a set of ordinals, then $A\in V[G\cap E_{e}]$ for some $E_{e}\in\mathcal{I}$.}
\end{lem}

Following the arguments in \cite[\textbf{Lemma 1.35}]{Dim2011}, we can see that in $V(G)$, $(\kappa'_{i})^{+}=\kappa_{i}$ for every $i<\omega$.
Similar to the arguments from the proof of \cite[\textbf{Theorem 11}]{AK2006}, it is possible to see that in $V(G)$, $\kappa=\aleph_{(\aleph_{\omega})^{V}}$ and $(\aleph_{\omega})^{V}=\aleph_{1}$. Consequently $\kappa=\aleph_{\omega_{1}}$ and $cf(\kappa)=\omega$ in $V(G)$. Further $\omega_{1}$ is singular in $V(G)$. Following \textbf{Fact 2.19}, $\mathsf{AC_{\omega}}$ fails in $V(G)$. We prove that an infinitary Chang conjecture holds in $V(G)$.

\begin{lem}
{\em In $V(G)$, an infinitary Chang conjecture holds.}
\end{lem}

\begin{proof}
Let $\mathcal{A}=\langle \kappa,...\rangle$ be a structure in a countable language in $V(G)$. 
Let $\{\phi_{n}: n<\omega\}$ be an enumeration of the formulas of the language of ${\mathcal{A}}$ such that each $\phi_{n}$ has $k(n)\leq n$ many free variables. Define $f:[\kappa]^{<\omega}\rightarrow 2$ by,

\centerline{$f(\epsilon_{1},...,\epsilon_{n})=1$ if and only if $\mathcal{A}\models\phi_{n}(\epsilon_{1},...,\epsilon_{k(n)})$ and $f(\epsilon_{1},...,\epsilon_{n})=0$ otherwise.}

By \textbf{Lemma 7.8}, there is an $E_{e}\in \mathcal{I}$ such that $f\in V[G\cap E_{e}]$. Fix an arbitrary $1\leq i< \omega$.
We can write $V[G\cap E_{e}]$=$V[G_1][G_2]$ where $G_1$ is $\mathbb{Q}_1$-generic over $V$ such that $\vert\mathbb{Q}_1\vert<\kappa_{i}$, and $G_2$ is $\mathbb{Q}_2$-generic over $V[G_1]$ such that $G_2$ adds no subsets of $\kappa_{i}$. Consequently, by \textbf{Lemma 7.6}, $\kappa_{i}\rightarrow^{\kappa_{i-1}}(\kappa_{i})^{<\omega}_{2}$ in $V[G\cap E_{e}]$. 
So, for all $1\leq i<\omega$, $\kappa_{i}\rightarrow^{\kappa_{i-1}}(\kappa_{i})^{<\omega}_{2}$ in $V[G\cap E_{e}]$. 

Let $e=\{\alpha_{1},...,\alpha_{m}\}$ where for each $i\in \{1,...,m\}$, there is a dictinct $\epsilon_{i}$ such that $\alpha_{i}\in (\kappa'_{\epsilon_{i}},\kappa_{\epsilon_{i}})$. Consider $j$ to be $max\{\epsilon_{i}:\alpha_{i}\in e\}$. If $G\cap E_{e}$ is $\mathbb{P}$-generic over $V$ then since $\vert \mathbb{P} \vert<\kappa_{j}$, 
by \textbf{Lemma 7.7}, $\langle \kappa_{i}:j-1\leq i<\omega\rangle$ is a coherent sequence of cardinals with the property $\kappa_{i}\rightarrow^{\kappa_{i-1}}(\kappa_{i})^{<\omega}_{2}$ for all $j\leq i<\omega$.
By \textbf{Definition 7.4}, there is a $\langle\kappa_{i}:j\leq i<\omega\rangle$-coherent sequence $\langle A_{n}:j\leq n<\omega\rangle$ of good indiscernibles for $\mathcal{A}$ with respect to $\langle\kappa_{i}:j-1\leq i<\omega\rangle$. 
We obtain a $\langle\kappa_{i}:j-1\leq i<\omega\rangle$-coherent sequence $\langle A_{n}:j-1\leq n<\omega\rangle$ of good indiscernibles for $\mathcal{A}$ with respect to $\langle\kappa_{i}:j-2\leq i<\omega\rangle$ as follows. 

\begin{itemize}
\item Since $\kappa_{j-1}\rightarrow^{\kappa_{j-2}}(\kappa_{j-1})^{<\omega}_{2}$, we obtain a set $A_{j-1}\in [\kappa_{j-1}\backslash \kappa_{j-2}]^{\kappa_{j-1}}$ of indiscernibles for $\mathcal{A}$ with respect to parameters below $\kappa_{j-2}$. Consequently, we obtain a $\langle\kappa_{i}:j-1\leq i<\omega\rangle$-coherent sequence $\langle A_{n}:j-1\leq n<\omega\rangle$ of good indiscernibles for $\mathcal{A}$ with respect to $\langle\kappa_{i}:j-2\leq i<\omega\rangle$.
\end{itemize}

If we continue in this manner step by step for the remaining cardinals $\kappa_{1},...\kappa_{j-2}$, then since $\kappa_{i}\rightarrow^{\kappa_{i-1}}(\kappa_{i})^{<\omega}_{2}$ for each $1\leq i\leq j-2$, we can obtain a $\langle\kappa_{i}:0 < i<\omega\rangle$-coherent sequence $A=\langle A_{n}:0< n<\omega\rangle$ of good indiscernibles for $\mathcal{A}$ with respect to $\langle\kappa_{i}: i<\omega\rangle$ and $A\in V[G\cap E_{e}]\subseteq V(G)$. 
Therefore for all $1\leq i<\omega$, $\kappa_{i}\rightarrow^{\kappa_{i-1}}(\kappa_{i})^{<\omega}_{2}$ and $\langle\kappa_{i}:1\leq i<\omega\rangle$ is a coherent sequence of cardinals in $V(G)$ by \textbf{Definition 7.4}. Using \textbf{Lemma 7.5}, we can obtain an infinitary Chang conjecture in $V(G)$ as \textbf{Lemma 7.5} can be proved in $\mathsf{ZF}$.
\end{proof}
\end{proof}

\begin{proof}{\textbf{(of Theorem 1.10)}}
Let $\mathcal{N}$ be the choiceless model constructed in \cite[\textbf{Theorem 11}]{AK2006}.
We first translate the arguments in terms of a symmetric extension based on a symmetric system $\langle \mathbb{P}, \mathcal{G}, \mathcal{F}\rangle$.

\begin{itemize}
    \item Consider $V$, $\mathbb{P}$, and $\mathcal{G}$ as mentioned in the previous construction (used for proving \textbf{Theorem 1.9}). 
    \item Let $\mathcal{I}=\{E_e : e\in \prod_{i< \omega}(\kappa_{i}', \kappa_{i})\}$ where for every $e=\{\alpha_i: i< \omega\}\in \prod_{i\in \omega}(\kappa_{i}', \kappa_{i})$, 
$E_{e}=\{\langle p_{i}\cap (\kappa_{i}'\times \alpha_{i}): i<\omega\rangle: \overrightarrow{p}\in \mathbb{P}\}$.
Let $\mathcal{I}$ generate a normal filter $\mathcal{F}_{\mathcal{I}}$ over $\mathcal G$. We define $\mathcal{F}$ to be $\mathcal{F}_{\mathcal{I}}$.
\end{itemize}

Let $G$ be a $\mathbb{P}$-generic filter. We consider the symmetric model $V(G)^{\mathcal{F}}$. We denote $V(G)^{\mathcal{F}}$ by $V(G)$ for the sake of convenience. The model $V(G)$ is analogous to the choiceless model $\mathcal{N}$ constructed in \cite[\textbf{Theorem 11}]{AK2006}. 
Since the forcing notions involved are weakly homogeneous, the following holds.

\begin{lem}
{\em If $A\in V(G)$ is a set of ordinals, then $A\in V[G\cap E_{e}]$ for some $E_{e}\in\mathcal{I}$.}
\end{lem}

Similar to \textbf{Lemma 7.9}, we observe an infinitary Chang conjecture in $V(G)$. 

\begin{lem}
{\em In $V(G)$, an infinitary Chang conjecture holds .}
\end{lem}

\begin{proof}
Let $\mathcal{A}=\langle \kappa,...\rangle$ be a structure in a countable language in $V(G)$. 
Let $\{\phi_{n}: n<\omega\}$ be an enumeration of the formulas of the language of $\mathcal{A}$ such that each $\phi_{n}$ has $k(n)\leq n$ many free variables. Define $f:[\kappa]^{<\omega}\rightarrow 2$ by,

\centerline{$f(\epsilon_{1},...,\epsilon_{n})=1$ if and only if $\mathcal{A}\models\phi_{n}(\epsilon_{1},...,\epsilon_{k(n)})$ and $f(\epsilon_{1},...,\epsilon_{n})=0$ otherwise.}

By \textbf{Lemma 7.10}, there is an $E_{e}\in \mathcal{I}$ such that $f\in V[G\cap E_{e}]$. Fix an arbitrary $1\leq i< \omega$.
We can write $V[G\cap E_{e}]$=$V[G_1][G_2]$ where $G_1$ is $\mathbb{Q}_1$-generic over $V$ such that $\vert\mathbb{Q}_1\vert<\kappa_{i}$, and $G_2$ is $\mathbb{Q}_2$-generic over $V[G_1]$ such that $G_2$ adds no subsets of $\kappa_{i}$. Consequently, by \textbf{Lemma 7.6}, $\kappa_{i}\rightarrow^{\kappa_{i-1}}(\kappa_{i})^{<\omega}_{2}$ in $V[G\cap E_{e}]$. 
So, for all $1\leq i<\omega$, $\kappa_{i}\rightarrow^{\kappa_{i-1}}(\kappa_{i})^{<\omega}_{2}$ in $V[G\cap E_{e}]$. 
Thus, we obtain a set $A_{i}\in [\kappa_{i}\backslash \kappa_{i-1}]^{\kappa_{i}}$ of good indiscernibles for $\mathcal{A}$ for each $1\leq i<\omega$, in $V[G\cap E_{e}]$. Consequently, we obtain a $\langle\kappa_{i}:0 < i<\omega\rangle$-coherent sequence $A=\langle A_{i}:0<i<\omega\rangle$ of good indiscernibles for $\mathcal{A}$ with respect to $\langle\kappa_{i}: i<\omega\rangle$ and $A\in V[G\cap E_{e}]\subseteq V(G)$. The rest is the same as in the proof of \textbf{Lemma 7.9}.
\end{proof}

Applying \cite[\textbf{Theorem 4}]{AK2006} and \cite[\textbf{Proposition 1}]{AK2008}, we prove that $\aleph_{\omega_{1}}$ is an almost Ramsey cardinal in  $V(G)$.

\begin{lem}
{\em In  $V(G)$, $\aleph_{\omega_{1}}$ is an almost Ramsey cardinal.}
\end{lem}

\begin{proof}
Following the terminologies from the proof of \cite[\textbf{Theorem 11}]{AK2006} we have $\kappa=\aleph_{\omega_{1}}$ in $V(G)$. We show that $\kappa$ is an almost Ramsey cardinal in $V(G)$.
Let $f:[\kappa]^{<\omega}\rightarrow 2$ be in $V(G)$. Since $f$ can be coded by a set of ordinals, $f\in V[G\cap E_{e}]$ for some $E_{e}\in\mathcal{I}$ by \textbf{Lemma 7.10}. Now, in $V$, $\kappa$ is the supremum of a coherent sequence of Ramsey cardinals $\langle\kappa_{i}: i<\omega\rangle$.
By \cite[\textbf{Theorem 4}]{AK2006} we can see that $\langle\kappa_{i}: i<\omega\rangle$ stays a coherent sequence of Ramsey cardinals in $V[G\cap E_{e}]$. 
Also $\kappa$ is the supremum of $\langle\kappa_{i}:i<\omega\rangle$ in $V[G\cap E_{e}]$.
Thus $\kappa$ is an almost Ramsey cardinal in $V[G\cap E_{e}]$ by \cite[\textbf{Proposition 1}]{AK2008}.
Thus for all $\beta<\kappa$, there is a set $X_{\beta}\in V[G\cap E_{e}]\subseteq V(G)$ which is homogeneous for $f$ and has order type at least $\beta$. Hence, $\kappa$ is almost Ramsey in $V(G)$ since $f$ was arbitrary.
\end{proof}
\end{proof}

\section{Mutual stationarity property from a sequence of measurable cardinals}
\subsection{Mutual stationarity}We recall the idea of {\em mutual stationarity} introduced by Foreman and Magidor in \cite{FM2001} and a theorem due to Cummings, Foreman, and Magidor from \cite{CFM2006}.

\begin{defn} {\textbf{(cf. \cite[\textbf{Definition 6}]{FM2001} $\&$  \cite[\textbf{Definition 1.1}]{Apt2005})}
Let $\mathcal{K}$ be a set of regular cardinals with supremum $\lambda$. Suppose that $S_{\kappa}\subseteq \kappa$ for each $\kappa\in\mathcal{K}$. Then $\langle S_\kappa : \kappa\in \mathcal{K}\rangle$ is {\em mutually stationary} if and only if for all algebras $\mathcal{A}$ on $\lambda$, there is an elementary substructure $\mathcal{B}\prec\mathcal{A}$ such that for all $\kappa\in \mathcal{B}\cap \mathcal{K}$, sup($\mathcal{B}\cap\kappa)\in S_{\kappa}$.}
\end{defn}

\begin{thm} {\textbf{(Cummings, Foreman,  and Magidor; cf. \cite[\textbf{Theorem 5.2}]{CFM2006})}}
{\em Let $\langle\kappa_{i}: i< \delta\rangle$ be an increasing sequence of measurable cardinals, where $\delta<\kappa_{0}$ is a regular cardinal. Let $S_{i}\subseteq \kappa_{i}$ be stationary for each $i<\delta$. Then $\langle S_i: i<\delta\rangle$ is mutually stationary.}
\end{thm}

\subsection{The Proof of Observation 1.11}
We recall the choiceless model constructed in  \cite[\textbf{Theorem 1}]{Apt1983a} and the terminologies from \cite{Apt1983a}. In particular we fix an arbitrary $n_{0}\in \omega$ and assume an increasing sequence of measurable cardinals $\langle \chi_{k}:k<\omega\rangle$ in the ground model $V$ of $\mathsf{ZFC}$. Then we consider the choiceless model constructed in \cite[\textbf{Theorem 1}]{Apt1983a}. For the sake of convenience we call the model $\mathcal{N}_{n_{0}}$ and recall the relevant lemmas from \cite{Apt1983a}.

\begin{lem}{(cf. \cite[\textbf{Lemma 1.1}]{Apt1983a})}
{\em If $X\in \mathcal{N}_{n_{0}}$ is a set of ordinals, then $X\in V[G\restriction f]$ for some $f\in K$.}
\end{lem}

\begin{lem}{(cf. \cite[\textbf{Lemma 1.2}]{Apt1983a})}
{\em Let $\lambda=\bigcup_{k<\omega}\chi_{k}$. Then $\lambda=\aleph_{\omega}$ in $\mathcal{N}_{n_{0}}$
}
\end{lem}

\begin{proof}{\textbf{(of Observation 1.11)}} We note that in $\mathcal{N}_{n_{0}}$,  $\chi_{k}=\aleph_{n_{0}+2(k+1)}$ for each $ k<\omega$. Let $\lambda=\bigcup_{k<\omega}\chi_{k}$ in $V$.
\begin{enumerate}
    \item Following the arguments in \cite[\textbf{Lemma 1.36}]{Dim2011}, $\aleph_{n_{0}+2(k+1)}$ is a measurable cardinal in $\mathcal{N}_{n_{0}}$, for each $1\leq k<\omega$. Following \cite[\textbf{Theorem 4.3}]{HK2019}, for each $1\leq k<\omega$, there are no uniform ultrafilters on $\aleph_{n_{0}+2k+1}$ in $\mathcal{N}_{n_{0}}$. Consequently for each $1\leq k<\omega$, $\aleph_{n_{0}+2k+1}$ can not be a measurable cardinal in $\mathcal{N}_{n_{0}}$.
    
\item We prove that in the model $\mathcal{N}_{n_{0}}$, if $\langle S_{k}: 1\leq k <\omega\rangle$ is a sequence of stationary sets such that $S_{k}\subseteq\chi_{k}$ for every $1\leq k <\omega$, then $\langle S_{k}: 1\leq k <\omega\rangle$ is mutually stationary. By \textbf{Lemma 8.4}, $\lambda=\aleph_{\omega}$ in $\mathcal{N}_{n_{0}}$.

Suppose $\mathcal{N}_{n_{0}}\models ``\mathcal{A}$ is an algebra on $\lambda$ and $\langle S_{k}: 1\leq k <\omega\rangle$ is a sequence of stationary sets such that $S_{k}\subseteq\chi_{k}$ for every $1\leq k <\omega$''.
Since both $\mathcal{A}$ and $\langle S_{k}: 1\leq k <\omega\rangle$ can be coded by set of ordinals, by \textbf{Lemma 8.3}, there exists some $f\in K$ for which both $\langle S_{k}: 1\leq k <\omega\rangle \in V[G\restriction f]$ and $\mathcal{A}\in V[G\restriction f]$. 

Following \cite[\textbf{Lemma 1.3}]{Apt1983a}, $\chi_k$ remains measurable in $V[G\restriction f]$ for every $1\leq k <\omega$. We can observe that $S_{k}$ is a stationary subset of $\chi_{k}$ in $V[G\restriction f]$ for every $1\leq k <\omega$. Fix any $1\leq k <\omega$. Let $C$ be any club subset of $\chi_{k}$ in $V[G\restriction f]$. Since the notion of club subset of $\chi_{k}$ is upward absolute and $V[G\restriction f]\subseteq \mathcal{N}_{n_{0}}$, $C$ is also a club subset of $\chi_{k}$ in $\mathcal{N}_{n_{0}}$. Since in $\mathcal{N}_{n_{0}}, S_{k}$ is a stationary subset of $\chi_{k}$ we have $S_{k}\cap C\not= \emptyset$. Thus, $S_{k}$ is a stationary subset of $\chi_{k}$ in $V[G\restriction f]$ for every $1\leq k <\omega$. By \textbf{Theorem 8.2}, $\langle S_{k}: 1\leq k <\omega\rangle$ is mutually stationary in $V[G\restriction f]$.

We note that $\mathcal{A}$ is an algebra on $\lambda$ in $V[G\restriction f]$.
Thus there is an elementary substructure $\mathcal{B}\prec\mathcal{A}$ in $V[G\restriction f]$ such that for all $k<\omega, sup(\mathcal{B}\cap\chi_{k})\in S_{k}$ by \textbf{Definition 8.1}. So there is an elementary substructure $\mathcal{B}\prec\mathcal{A}$ in $\mathcal{N}_{n_{0}}$ such that for all $k<\omega, sup(\mathcal{B}\cap\chi_{k})\in S_{k}$. Hence in $\mathcal{N}_{n_{0}}$, $\langle S_{k}: 1\leq k <\omega\rangle$ is mutually stationary.  

\item  We recall that $\lambda=\aleph_{\omega}$ in $\mathcal{N}_{n_{0}}$. 
We can see that $\lambda$ is an almost Ramsey cardinal in $\mathcal{N}_{n_{0}}$ by a well-known argument from \cite[\textbf{Lemma 2.5}]{ADK2016}. For reader's convenience, we provide a sketch of the proof.
Let $f:[\lambda]^{<\omega}\rightarrow 2$ be in $\mathcal{N}_{n_{0}}$. Since $f$ can be coded by a set of ordinals, $f\in V[G\restriction f]$ for some $f\in K$ by \textbf{Lemma 8.3}. Following \cite[\textbf{Lemma 1.3}]{Apt1983a}, $\chi_k$ remains measurable in $V[G\restriction f]$ for every $1\leq k <\omega$. Consequently, $\chi_{k}$ is Ramsey in $V[G\restriction f]$ for every $1\leq k <\omega$. Now, in $V[G\restriction f]$, $\lambda$ is the supremum of Ramsey cardinals $\langle\chi_{i}: 1\leq i<\omega\rangle$.
Thus $\lambda$ is an almost Ramsey cardinal in $V[G\restriction f]$ by \cite[\textbf{Proposition 1}]{AK2008}.
Thus for all $\beta<\lambda$, there is a set $X_{\beta}\in V[G\restriction f]\subseteq \mathcal{N}_{n_{0}}$ which is homogeneous for $f$ and has order type at least $\beta$. Hence, $\lambda$ is almost Ramsey in $\mathcal{N}_{n_{0}}$ since $f$ was arbitrary.
\end{enumerate}
\end{proof}

\section{Acknowledgements}
The author would like to thank the reviewer for reading the manuscript in details and providing several suggestions for improvement.
The author would like to thank Arthur Apter for communicating the homogeneous property of the strongly compact Prikry forcing mentioned in \cite[\textbf{Lemma 2.1}]{AH1991} as well as
for the conversations concerning \textbf{Remark 3.3} and \textbf{Remark 3.13}. 
The author would like to thank Asaf Karagila for helping to translate the arguments of Apter from \cite[\textbf{Theorem 1}]{Apt2001} in terms of symmetric extension by a symmetric system $\langle \mathbb{P},\mathcal{G},\mathcal{F}\rangle$. We constructed a  similar symmetric extension to prove \cite[\textbf{Theorem 4.1}]{Kar2019}.

\textbf{Statements and Declarations:}

\textbf{Funding}: 

No funding was received for conducting this study.

\textbf{Conflict of interest}: 

We declare no conflicts of interest.

\end{document}